\documentclass[a4paper, 11pt]{article}

\usepackage[a4paper,left=26.5mm,right=26mm,top=26.25mm,bottom=36.75mm]{geometry}

\usepackage[english]{babel}	% english could be changed for french
\usepackage{lmodern}
\usepackage[T1]{fontenc}
\usepackage[utf8]{inputenc}	% In order to display accents
\usepackage{amsfonts, amsmath, amsthm, amssymb} %AMS packages
\usepackage{mathtools}
\usepackage{mathrsfs}	% pour pouvoir utiliser \mathscr{}
\usepackage{bbm}	%for the symbol one or II for the identity map $\mathbbm{1}$ or $\mathbbm{I}$
\DeclareMathAlphabet{\mathbbold}{U}{bbold}{m}{n}	% nuovo alfabeto per i numeri in \usepackage{verbatim}

\usepackage[leftbars]{changebar}
\usepackage{comment}
\usepackage{hyperref}

\usepackage{authblk} % Pour les auteurs multiples et leurs affiliations
\usepackage[strings]{underscore}

\usepackage{enumitem}
\usepackage{subfig}
\usepackage{array}
\usepackage{color}
\usepackage{stmaryrd}
\usepackage{dsfont}
\usepackage{pdflscape}
\usepackage{datetime}
\usepackage{cases}
\usepackage{upgreek}
\usepackage[nocompress]{cite}

\usepackage{doi}

\usepackage{thmtools}
\usepackage{thm-restate}

\newcommand{\R}{\mathbb{R}}
\newcommand{\Z}{\mathbb{Z}}

\newcommand{\N}{\mathbb{N}}

\renewcommand{\leq}{\leqslant}
\renewcommand{\geq}{\geqslant}

\newcommand{\eps}{\varepsilon}
\renewcommand{\epsilon}{\varepsilon}

\newcommand{\T}[1]{{1}^\mathrm{T}}

\newcommand*\diff{\mathop{}\!\mathrm{d}}

\newcommand{\cv}{\to}
\newcommand{\cvl}[1]{\underset{#1}{\longrightarrow}}

\newcommand{\ind}{\mathbbm{1}}

\DeclareMathOperator{\argmin}{argmin}

\newtheorem{theorem}{Theorem}[section]
\newtheorem{corollary}[theorem]{Corollary}
\newtheorem{lemma}[theorem]{Lemma}
\newtheorem{proposition}[theorem]{Proposition}

\theoremstyle{definition}

\newtheorem{remark}[theorem]{Remark}

\newcommand{\cld}{q}

\newcommand{\noy}{k}
\newcommand{\rmin}{r_{\min}}
\newcommand{\rmax}{r_{\max}}

\newcommand{\lmax}{\ell_{\max}}

\newcommand{\conc}{\kappa}

\DeclareMathOperator{\diag}{diag}

\newcommand{\dtc}{\psi}
\newcommand{\cont}{u}
\newcommand{\vit}{G}

\newcommand{\tmax}{t_{\max}}

\newcommand{\regis}{\textsuperscript{\mbox{\scriptsize{\textregistered}}} }

\newcommand{\ee}{\eta}
\newcommand{\rp}{\rho_{z}}
\newcommand{\rt}{\rho_{y}}

\renewcommand{\Im}{\text{Im }}

\newcommand{\fonction}[5]{
\begin{align*}
\displaystyle
\begin{array}{lrcl}
#1: & #2 & \longrightarrow & #3 \\
    & #4 & \longmapsto & #5
\end{array}
\end{align*}}

\newcommand{\etath}{\hat{\dtc}}

\newcommand{\opc}{\mathcal{K}}

\title{New inversion methods for the single/multi-shape CLD-to-PSD problem with spheroid particles}
\author[1]{Lucas Brivadis}
\author[1]{Ludovic Sacchelli}
\affil[1]{Univ. Lyon, Universit\'e Claude Bernard Lyon 1, CNRS, LAGEPP UMR 5007, 43 bd du 11 novembre 1918, F-69100 Villeurbanne, France (email: \texttt{lucas.brivadis@univ-lyon1.fr}, \texttt{ludovic.sacchelli@univ-lyon1.fr})}

\date{\today}

\begin{document}

\maketitle

\begin{abstract}
    In this paper, we express the Chord Length Distribution (CLD) measure associated to a given Particle Size Distribution (PSD) when particles are modeled as suspended spheroids in a reactor.
    Using this approach, we propose two methods to reconstruct the unknown PSD from its CLD.
    In the single-shape case where all spheroids have the same shape, a Tikhonov regularization procedure is implemented.
    In the multi-shape case, the measured CLD mixes the contribution of the PSD associated to each shape.
    Then, an evolution model for a batch crystallization process
    allows to introduce a Back and Forth Nudging (BFN) algorithm, based on dynamical observers.
    We prove the convergence of this method when crystals are split into two clusters: spheres and elongated spheroids.
    These methods are illustrated with numerical simulations.
\end{abstract}

\noindent {\bf Keywords:}
\begin{minipage}[t]{.8\linewidth}
\flushleft
Particle Size Distribution,
Chord Length Distribution,
spheroid particles,
Tikhonov regularization,
Back and Forth Nudging.
\end{minipage}

% \begin{keywords}
% Particle Size Distribution,
% Chord Length Distribution,
% spheroid particles,
% Tikhonov regularization,
% Back and Forth Nudging.
% \end{keywords}

\section{Introduction}

Estimating the size distribution of suspended particles in a reactor is a major issue in process control \cite{kleinboistelle}.
In crystallization processes,
the particles shape and size govern important physico-chemical properties of the product, hence the need to be controlled and estimated.
Measuring a Particle Size Distribution (PSD) remains a challenging problem,
tackled by modern Process Analytical Technologies (PATs) with various measures and approaches, such as image processing \cite{Benoit, GAO}, dynamical observers based on solute concentration \cite{BRIVADIS202017} or moments based methods \cite{vissers2012model, porru2017monitoring, mesbah2011comparison, Uccheddu, lebaz, gruy:hal-01637703}.
In this paper, we focus on PATs giving access to Chord Length Distributions (CLDs), such as the Focused Beam Reflectance Measurement (FBRM) or the BlazeMetrics\regis technologies.
In the recent years, several strategies have been proposed to recover the PSD from the knowledge of its corresponding CLD \cite{worlitschek2005restoration, liu1998relationship, pandit2016chord, AGIMELEN2015629}.

Understanding the CLD-PSD relation is an essential step in recovering the desired PSD when using the above-mentioned technologies. Naturally, this relation is heavily influenced by the shape of the particles.
In \cite{LANGSTON200133, barrett, hobbel1991modern, brivadis:hal-01900402}, the authors considered spherical particles. It often occurs in crystallization processes that particles cannot be assumed to have such symmetries.
In \cite{AGIMELEN2015629} for instance, needle-shaped particles were modeled as cylinders.
In section~\ref{sec2}, we propose a new model for computing the CLD from a PSD of spheroid-like particles.
Spheroids are generalized sphere-like shapes that have the advantage of allowing to model both spheres and elongated particles with only one shape tuning parameter.
In that respect, we gather different shapes under the same mathematical umbrella while retaining many computational properties of the spherical model.
%, which proves useful to address the issues of polymorphism.
Note that, unlike \cite{LI20053251} who considered two-dimensional ellipses, we consider proper three-dimensional spheroids that can be measured by the probe in any possible orientation. Spheroids were already considered in \cite{kellerer1984chord}, but the experimental assumptions lead to differing probabilistic models and distributions.

In the single-shape case (section~\ref{sec3}) composed of spheroid-like particles,
we prove the PSD-to-CLD relation to be one-to-one.
Hence, recovering the PSD from the CLD is an inverse problem that can be solved by a direct method. However, the problem is ill-posed: small perturbations of the measured CLD may induce large variations of the reconstructed PSD. We apply a regularization procedure, known as Tikhonov regularization, in order to ensure robustness with respect to measurement noise of the CLD.

Due to polymorphism in crystallization processes, it frequently occurs that particles in the reactor have not only different sizes, but also different shapes.
These different shapes correspond to stable and/or metastable phases that appear, grow and may disappear during the process \cite{GAO, Mullin, Mersmann}.
In this context (see section~\ref{sec4}), estimating the PSD of each shape only from the knowledge of the shared CLD (the sum of the CLDs associated to each PSD) is a much more difficult issue, which, to the best of our knowledge, has not been investigated from a theoretical viewpoint.
In this paper we propose, in the multi-shape case, to make use not only of a measure of the CLD, but also of an evolution model of the PSD, to estimate the PSD with observer techniques.
Observers have proved to be very useful in the context of batch crystallization \cite{Uccheddu, nagy2013recent, motz2008state, mesbah2011comparison, vissers2012model, porru2017monitoring}. Here we apply the Back and Forth Nudging (BFN) algorithm
\cite{brivadis:hal-02529820, haine2011, haine2014recovering, haine2011fr, auroux2005back, auroux2008nudging, auroux2009back, auroux2012back, Ramdani, ito2011time}, which is an inverse problem technique based on dynamical observers.
We prove the convergence of this method when crystals are split into two clusters: spheres and elongated spheroids, which 
happens, for example, in \cite{GAO}.

For the two inversion procedures (Tikhonnov regularization and BFN algorithm), we provide a theoretical analysis and numerical simulations.

\paragraph{Notations.} In the paper, the notations $C^k, L^p, H^p$ denote functional spaces pertaining to properties of the distributions. $C^k(I, X)$ functions are $k$-continuously differentiable functions from $I$ to $X$. $L^p(I, X)$ functions are such that their $p$-power is integrable over $I$. $H^p(I, X)$ denotes the space of $p$-differentiable functions from $I$ to $X$ such that each of the derivatives of degree lesser or equal to $p$ is also $L^2(I,X)$.
We denote by $\mathbb{P}(A)$ the probability of event $A$, and $\mathbb{E}(R)$ the expectation of random variable $R$.
The partial derivative of a function $f$ with respect to the variable $x$ is denoted by $\frac{\partial f}{\partial x}$.

\section{From PSD to CLD for spheroids}\label{sec2}

\subsection{From spheroid geometry to chord length}\label{sec:ellips}

In this paper, we are concerned with particles whose shape can be approximated by a spheroid (also called ellipsoid of revolution).
A spheroid is a surface of revolution, obtained as the rotation of an ellipse along one of its two principal axes.
In particular, spheres are spheroids.
When scanning across some particles, the sensor measures chords on the projection of the particle on the plane that is orthogonal to the probe's laser beam.
Hence, two sources of hazards must be considered to model the random choice of the chords measured by the sensor:
\begin{itemize}
    \item choice of orientation of the spheroid with respect to the probe;
    \item choice of the chord on the projection of the spheroid with selected orientation.
\end{itemize}

\paragraph{Step 1: Choosing an orientation.}
A spheroid of radius $r$ in elementary orientation can be represented as the set of points $(x,y,z)\in \R^3$ such that
\begin{equation}\label{spheroid_elem}
\begin{pmatrix}
    x & y & z
\end{pmatrix}
D
\begin{pmatrix}
    x \\ y \\ z
\end{pmatrix}
\leq r^2
\quad 
\text{ with }
\quad
D=
\begin{pmatrix}
    1 & 0 & 0
    \\
    0 & 1 & 0
    \\
    0 & 0 & \frac{1}{\ee^2}
\end{pmatrix}.
\end{equation}
The parameter $\ee$ is the ratio of the diameter of the spheroid along the axis of rotation by the diameter perpendicular to this axis.  
It characterizes the eccentricity of the spheroid.
The spheroid is said to be prolate if $\ee>1$ and oblate if $\ee<1$. When $\ee=1$, the particle is a sphere.
The volume of such a particle is given by $\frac{4\pi}{3}\ee r^3$.

Without loss of generality, we assume that the probe's laser beam is parallel to the $z$-axis. The solid can be oriented in any direction in space.
Since the solid is a spheroid, it has an axis of symmetry and any orientation is equivalent to picking a point on the sphere in 3d space, corresponding, for instance, to the position of the north pole of the spheroid (see Figure \ref{fig:3d}).
For this reason, we obtain an orientation following spherical coordinates.
Hence a sequence of two rotations of the elementary spheroid \eqref{spheroid_elem} allows to choose any possible orientation.
\begin{itemize}
    \item First, we rotate the space around the $y$-axis with an angle $\theta\in [0,\pi]$, leading to a change of coordinates of the matrix
    $$
    \rt(\theta)=
    \begin{pmatrix}
        \cos \theta & 0& \sin \theta
        \\
        0 &  1 & 0
        \\
        -\sin \theta & 0 &\cos \theta
    \end{pmatrix}.
    $$
    \item Second, we rotate the space around the $z$-axis with an angle $\phi\in [0,2\pi]$, leading to a change of coordinates of the matrix 
    $$
    \rp(\phi)=
    \begin{pmatrix}
        \cos \phi & -\sin \phi & 0
        \\
        \sin \phi & \cos \phi & 0
        \\
        0 & 0 & 1
    \end{pmatrix}.
    $$
\end{itemize}
The change of coordinates $(x,y,z)^\top\mapsto \rp(\phi)\rt(\theta)(x,y,z)^\top$ has the effect of mapping the point $(0,0,1)$ to any point on the sphere. Furthermore, it is an isometry.
% This has the effect that the change of coordinates given by $R(\phi,\theta)$ maps the elemnetary ellipse onto the same ellipse with another orientation.
If  $(\phi,\theta)\in [0,2\pi]\times [0,\pi]$ is picked according to the probability measure $\diff \mu= \frac{\sin \theta}{4\pi} \diff \phi\diff \theta$, this equals to \emph{uniformly} picking a random orientation for the spheroid (that is, the measure $\mu$ gives a uniform probability of picking a point on the sphere).
Then, the change of coordinates implies that the rotated spheroid has  equation
\begin{equation}\label{spheroid_rot}
\begin{pmatrix}
    x & y & z
\end{pmatrix}
A
\begin{pmatrix}
    x \\ y \\ z
\end{pmatrix}
\leq r^2
\quad
\text{ with }
\quad
A = \rp(\phi)\rt(\theta)
\,D\,
\rt(-\theta)\rp(-\phi).
\end{equation}
In Figure~\ref{fig:3d}, the elementary spheroid \eqref{spheroid_elem} (on the left) is rotated by $\rp(\phi)\rt(\theta)$ to obtain the rotated spheroid \eqref{spheroid_rot} (on the right).

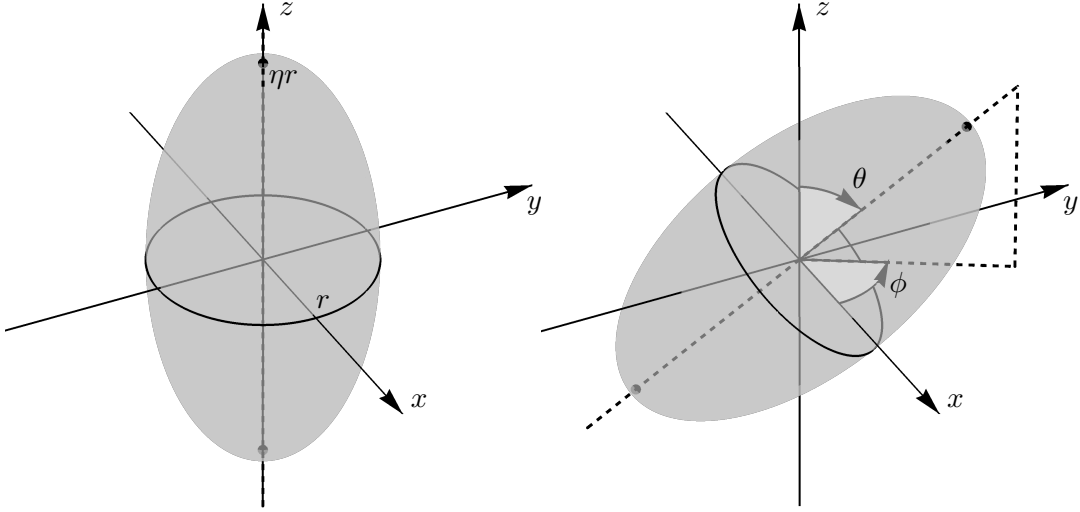
\begin{figure}[ht!]
    \centering
    \setlength{\unitlength}{.44\linewidth}
    \begin{picture}(1,1)
        \put(0,0){\includegraphics[width=\unitlength,trim=5cm 7.5cm 5.cm 6.5cm,clip]{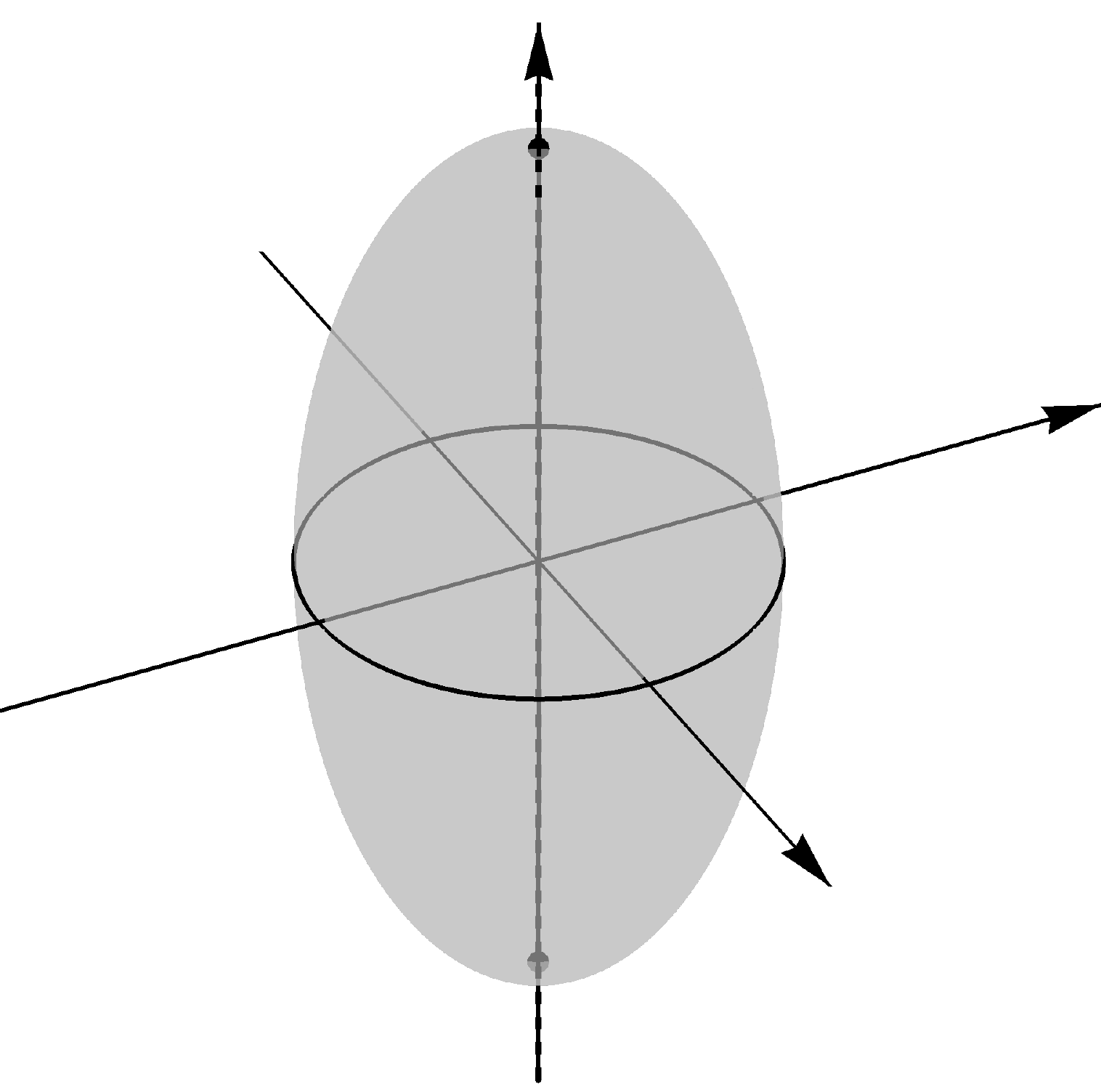}}
        \put(.77,.2){$x$}
        \put(.99,.58){$y$}
        \put(.52,.95){$z$}
        \put(.59,.39){$r$}
        \put(.5,.82){$\ee r$}
    \end{picture}
%
    %\hspace{.1\linewidth}
%
    \begin{picture}(1,1)
        \put(0,0){\includegraphics[width=\unitlength,trim=5cm 7.5cm 5.cm 6.5cm,clip]{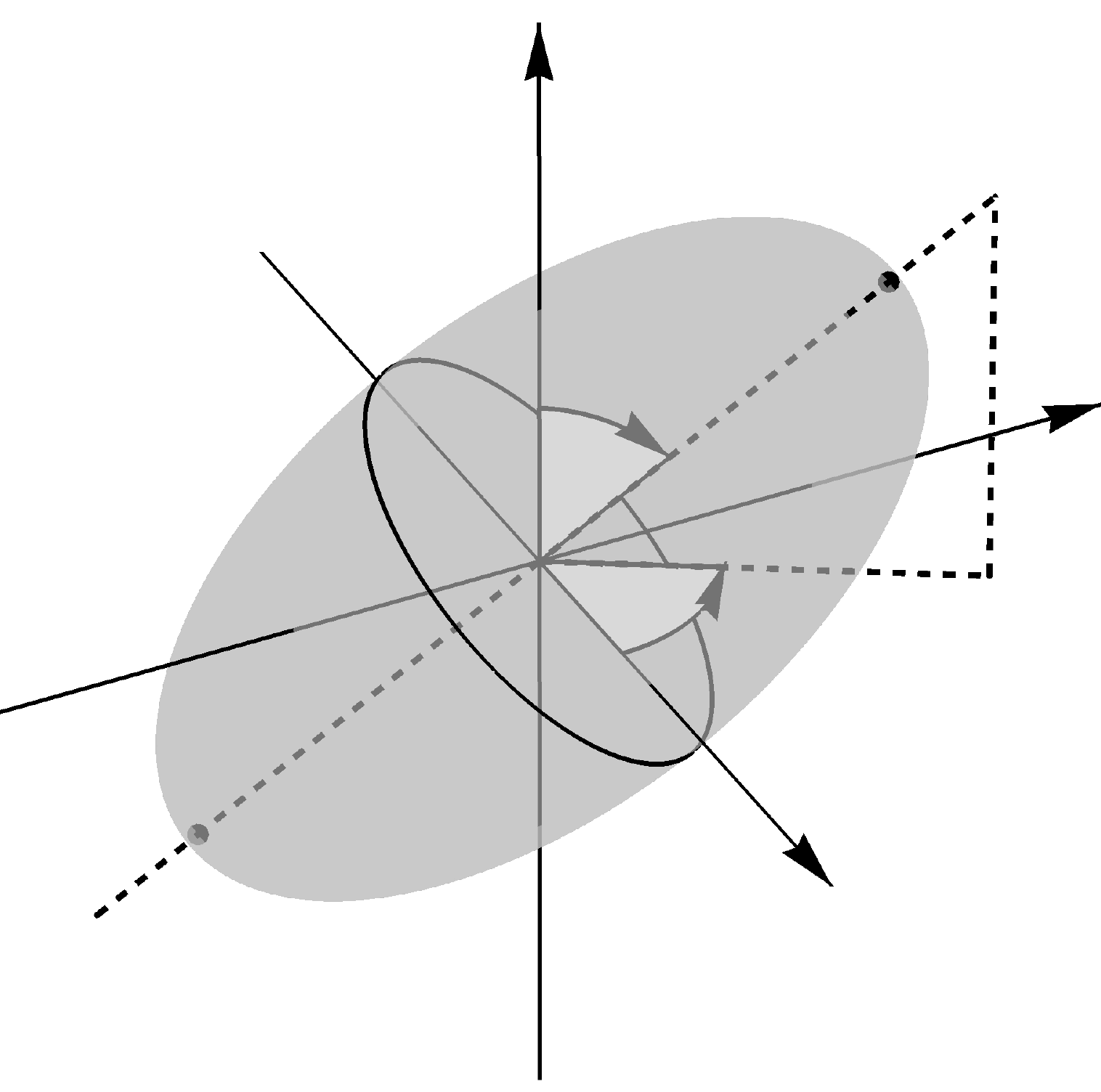}}
        \put(.77,.2){$x$}
        \put(.99,.58){$y$}
        \put(.52,.95){$z$}
        \put(.66,.42){$\phi$}
        \put(.59,.62){$\theta$}
    \end{picture}

    \caption{\emph{On the left:} elementary spheroid of parameter $\ee$ and radius $r$ (equation \eqref{spheroid_elem}).
    \emph{On the right:} rotation of the elementary spheroid with angles $\phi$, $\theta$ (equation \eqref{spheroid_rot})}
    \label{fig:3d}

\end{figure}

\paragraph{Step 2: Projecting the spheroid on the $(x,y)$-plane.} Given an arbitrary orientation of the particle in space, the sensor measure (assumed parallel to the $z$-axis) is the same as the one given by the ellipse obtained by projection of the solid on the $(x,y)$-plane. Hence, the next step is to transfer the geometry of the 3d spheroid onto its shadow in the $(x,y)$-plane.
The shell of the spheroid is given by $(x,y,z)A(x,y,z)^\top=r^2$ for some $(\phi,\theta)\in [0,2\pi]\times [0,\pi]$.
For completeness sake, the full expression of matrix $A$ is the following:
$$
\begin{pmatrix}
\bar s \cos ^2\phi  +\sin ^2\phi  
 & 
 -\bar\ee \sin ^2\theta  \sin2 \phi   
 &
-\bar \ee \sin 2\theta   \cos \phi 
 \\
 -\bar\ee \sin ^2\theta  \sin2 \phi
 & 
\bar s  \sin ^2\phi  +\cos ^2\phi  
 &
 -\bar\ee  \sin2 \theta   \sin \phi 
 \\
-\bar \ee \sin 2\theta   \cos \phi
 & 
 -\bar\ee  \sin2 \theta   \sin \phi 
 & 
\frac{1}{\ee^2}\cos^2\theta +\sin^2\theta    
\end{pmatrix}
\quad\text{ with }\quad
\begin{matrix}
\bar{\ee}=\frac{\ee^2-1}{2\ee^2},
\vspace{.2cm}
\\
\bar{s}=\frac{\sin^2\theta+\ee^2\cos^2\theta}{\ee^2}.
\end{matrix}
$$

If we are looking at points that appear at the edge of the shadow of the spheroid, it is clear that these must be such that the tangent plane to the spheroid at that point is vertical (see Figure~\ref{F:projection}).

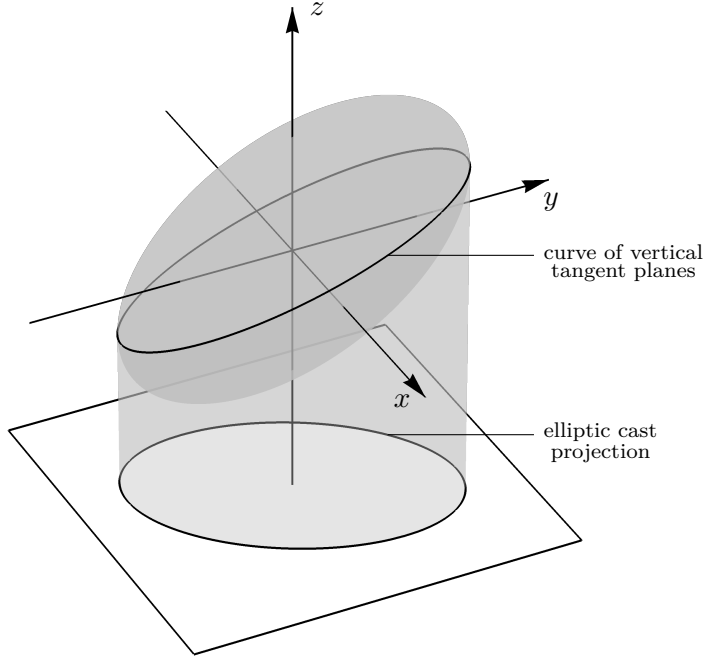
\begin{figure}[ht!]
    \centering
        \setlength{\unitlength}{.5\linewidth}
    \begin{picture}(1.1,1.1)
        \put(0,0){\includegraphics[width=\unitlength,trim=3.5cm 2cm 3.5cm 6.5cm,clip]{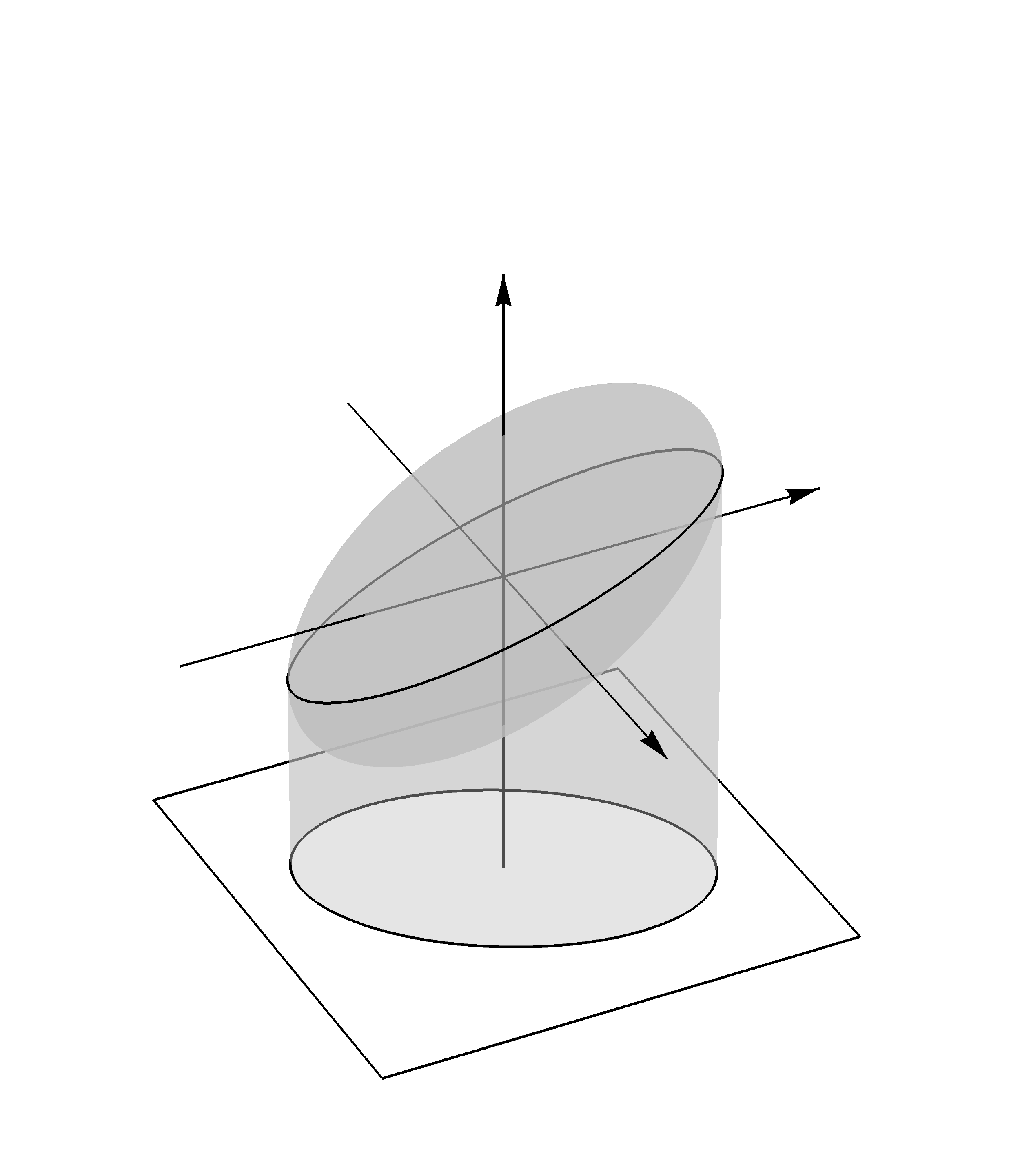}}
        \put(.66,.44){$x$}
        \put(.91,.78){$y$}
        \put(.52,1.1){$z$}
        \put(.65,.67){\rule[.5em]{5em}{.4pt} $\substack{\text{curve of vertical} \\ \text{tangent planes}}$}
        \put(.65,.368){\rule[.5em]{5em}{.4pt} $\substack{\text{elliptic cast} \\ \text{projection}}$}
    \end{picture}
    
    \caption{Projection of a spheroid on the $(x, y)$-plane.}
    \label{F:projection}
\end{figure}

Since the spheroid is given by an implicit definition of the form $g(x,y,z)=r^2$, the tangent plane to the spheroid at a point $(x,y,z)$  is actually the plane that is orthogonal to $\nabla g(x,y,z)$, the gradient of $g$ at $(x,y,z)$.
Hence, to find points $(x,y)$ in the plane that lie at the border of the shadow cast by the spheroid, we solve
$$
    g(x,y,z)=r^2,
    \qquad
    \left(\nabla g(x,y,z)\right)^\top \begin{pmatrix}
    0 \\ 0 \\ 1
\end{pmatrix}=0.
$$

In the case of a spheroid, $g(x,y,z)=(x,y,z)A(x,y,z)^\top$, hence $\nabla g(x,y,z)=A\cdot(x,y,z)^\top$.
In other words, we solve 
$$
    \begin{pmatrix}
    x & y & z
\end{pmatrix}
A
\begin{pmatrix}
    x \\ y \\ z
\end{pmatrix}
= r^2,\qquad
    \begin{pmatrix}
    0 & 0 & 1
\end{pmatrix}
A
\begin{pmatrix}
    x \\ y \\ z
\end{pmatrix}
= 0.
$$
In the $(x,y)$-plane, solutions to this pair of equations are points of the planar ellipse
\begin{equation}\label{E:ellipse}
    \alpha x^2+\beta y^2+\gamma xy=r^2,
\end{equation}
with
\begin{align}
\alpha&=\frac{\cos ^2\phi}{\cos^2\theta+\ee^2\sin^2 \theta}+\sin ^2\phi,\label{E:alpha}
\\
\beta&=\frac{\sin ^2\phi}{\cos^2\theta+\ee^2\sin^2 \theta}+\cos ^2\phi,
\\
\gamma&=-\frac{ (\ee^2-1) \sin ^2 \theta  \sin 2 \phi }{\cos^2\theta+\ee^2\sin^2 \theta }.
\end{align}
Naturally, $\cos^2\theta+\ee^2\sin^2 \theta >0$ for all $\ee>0$ and $\theta\in [0, \pi]$.
In conclusion, the shadow cast by the spheroid has the shape of an ellipse of orientation and eccentricity determined by the quantities $\alpha,\beta,\gamma$, themselves functions of $\phi,\theta$ and $\ee$.
When necessary, we write $\alpha_\ee$ to underline the $\ee$-dependence.

\paragraph{Step 3: Choosing a chord on the projection.}

Since we considered all the possible orientations of the spheroid in space, we can consider with no loss of generality that the probe's laser cut the two-dimensional projection \eqref{E:ellipse} at constant $y$.
Hence, the length of a chord on \eqref{E:ellipse} at some constant $y\in\R$ is the distance 
between the two $x$-solutions, if they exist, of 
\begin{equation}\label{E:trinome}
    \alpha x^2+\gamma y x+\beta y^2-r^2=0.
\end{equation}
Let $\Delta=\gamma^2 y^2-4 \alpha (\beta y^2-r^2)$ be the discriminant of the quadratic and let \begin{equation}\label{E:def_ymax}
    y_{\max} =\dfrac{2\sqrt{\alpha} r}{\sqrt{4\alpha \beta -\gamma^2}}.
\end{equation}
(Let us precise that $4\alpha \beta -\gamma^2=\frac{8 }{1+\ee^2+(\ee^2-1) \cos 2 \theta}>0$ for all $\ee>0$ and all $\theta\in[0, \pi]$.)
If $|y|\leq y_{\max}$, then $\Delta\geq0$ and the length of the chord cutting \eqref{E:ellipse} at $y$ is given by
\begin{equation}\label{E:ellDelta}
\ell = \frac{\sqrt{\Delta}}{\alpha}.    
\end{equation}
Otherwise, if $|y|> y_{\max}$, \emph{i.e.}, $\Delta<0$,
then no chord cuts the ellipse at $y$.
Hence, the maximum chord length is
$\frac{2r}{\sqrt{\alpha}}$, reached at $y = 0$.
For all $\ell\in[0, \frac{2r}{\sqrt{\alpha}}]$, let $y_\ell$ be such that the chord length $\ell$ is reached at $y=y_\ell$, so that $y_\ell$ is implicitly defined by \eqref{E:ellDelta}:
\begin{equation}\label{E:def_yell}
    y_{\ell} =\dfrac{\sqrt{4\alpha r^2-\alpha^2\ell^2} }{\sqrt{4\alpha \beta -\gamma^2}}.
\end{equation}
If $\ell>\frac{2r}{\sqrt{\alpha}}$, adopt the convention $y_\ell=0$. Doing so, $y_\ell$ is a continuous function of $\ell$.
These notations are summarized in Figure~\ref{F:cordes}.

To conclude, for a given spheroid of radius $r$ and ratio $\ee$ with orientation $(\phi, \theta)$ in space, the chord length $\ell$ is measured by the sensor when cutting the projection of the particle on the $(x, y)$-axis at constant $y = y_{\ell}$.

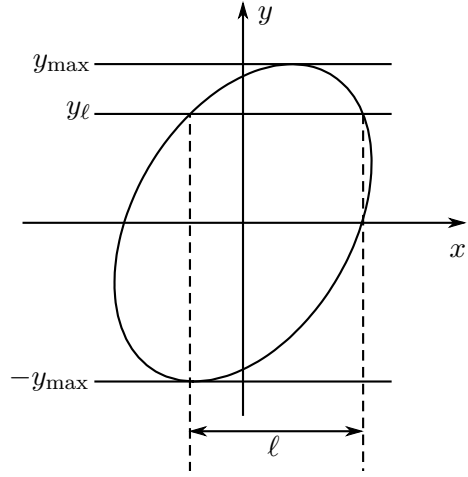
\begin{figure}[ht!]
    \centering
    \setlength{\unitlength}{.4\linewidth}
    \begin{picture}(1,1)
        \put(0,0){\includegraphics[width=\unitlength]{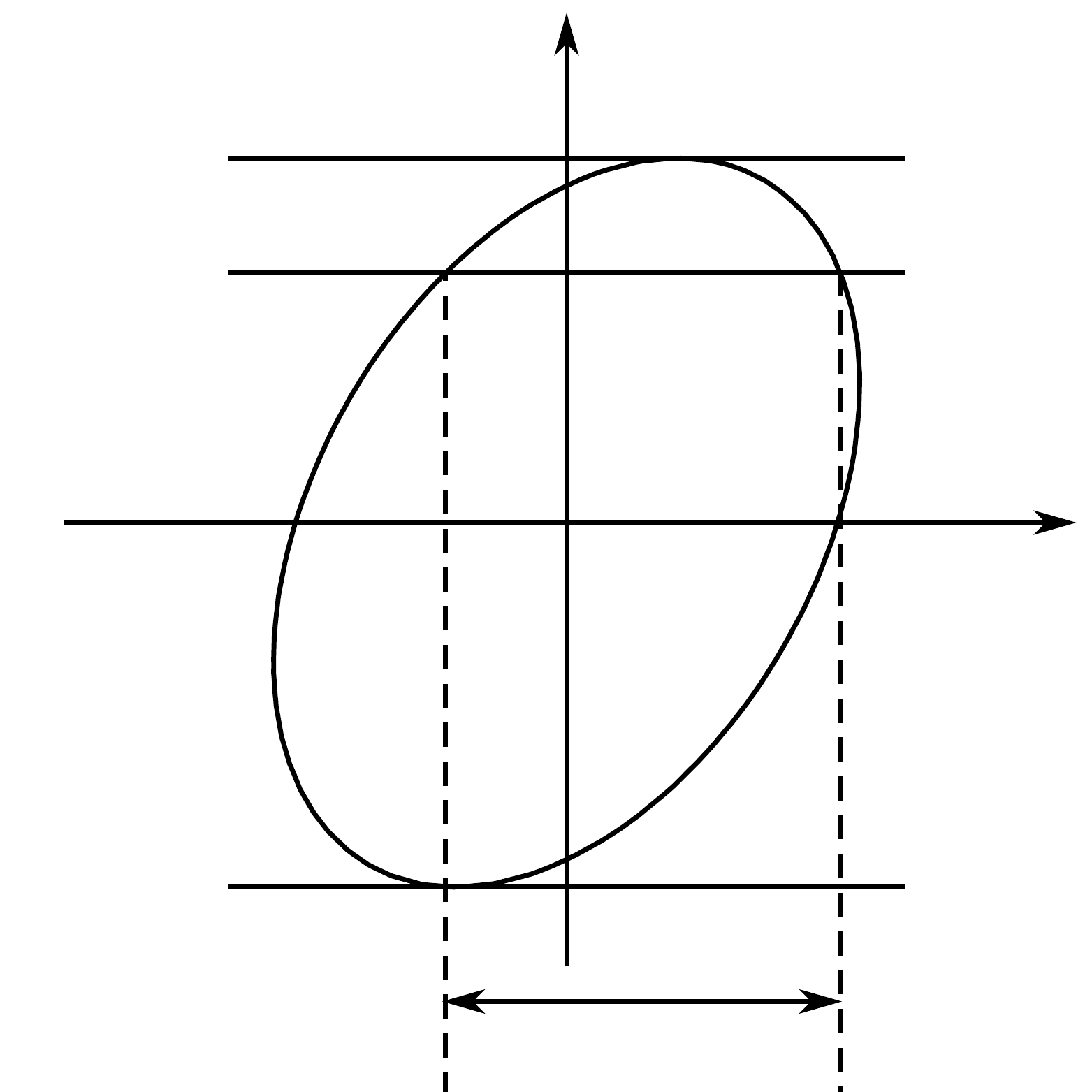}}
        \put(.55,.95){$y$}
        \put(.95,.45){$x$}
        \put(.57,0.03){$\ell$}
        \put(.08,0.85){$y_{\max}$}
        \put(.15,0.75){$y_\ell$}
        \put(.03,0.18){$-y_{\max}$}
%         \put(-.013,0){
% 	\begin{tikzpicture}[x=\unitlength,y=\unitlength]
% 		\draw[step=.1\unitlength,gray,thin] (0,0) grid (\unitlength,\unitlength);
% 	\end{tikzpicture}
% 		}
    \end{picture}
    \caption{Length $\ell$ of an horizontal chord on an ellipse at $y = y_\ell\in[-y_{\max}, y_{\max}]$.}
    \label{F:cordes}
\end{figure}

\subsection{From spheroid distribution to cumulative CLD}

Let us denote by $\psi(r)$ (in $\text{m}^{-1}.\text{m}^{-3}$) a PSD of spheroids of parameter $\ee$ (dimensionless) and radius $r$ (in m) between $\rmin$ and $\rmax$, generating a CLD measured by the sensor in a batch reactor.
For $r_1<r_2$,
the integral
$\int_{r_1}^{r_2}\dtc(r)\diff r$
(in $\text{m}^{-3}$) represents the number of particles with radius $r$ between $r_1$ and $r_2$ per unit of volume.
The corresponding CLD is denoted by $\cld(\ell)$ (in $\text{m}^{-1}.\text{m}^{-3}$).
Note that the largest possible chord of a spheroid of radius $r$ is the diameter of the spheroid, namely, $\lmax =2\rmax\max(\ee, 1)$.
Hence $0\leq\ell\leq \lmax$. Then $\int_{\ell_1}^{\ell_2}\cld(\ell)\diff \ell$ represents the number of chords with length $\ell$ between $\ell_1$ and $\ell_2$ measured by the sensor per unit of volume.
The cumulative CLD is denoted by $Q(\ell) = \int_{0}^{\ell} q(l)\diff l$ (in m$^{-3}$).
Then, the renormalized functions
$\bar \psi(r) = \frac{1}{\int_{\rmin}^{\rmax} \psi(\rho)\diff \rho} \psi(r)$ and
$\bar q(\ell) = \frac{1}{Q(\lmax)} q(\ell)$ are probability density functions (in m$^{-1}$)
and $\bar Q(\ell) = \frac{1}{Q(\lmax)} Q(\ell)$ is a cumulative distribution function
(dimensionless).

Let $R$ be a random variable representing the radius of a particle,
and $L$ be a random variable representing a measured chord length.
By law of total expectation,
\begin{equation}\label{nbtoQb}
    \bar Q(\ell) := \int_0^\ell \bar q (l)\diff l
    =
    \mathbb{P}(L < \ell)
    = \mathbb{E}(\ind_{L\leq \ell})
    = \mathbb{E}(\mathbb{E}(\ind_{L\leq \ell} | R))
    = \int_{\rmin}^{\rmax} \noy(\ell, r)\bar \psi(r)\diff r.
\end{equation}
where
$$
\noy(\ell,r)=\mathbb{P}(L<\ell \mid R=r)
$$
encodes the probability of measuring a chord length less than $\ell$ assuming a particle of radius $r$ crosses the sensor.
Hence
\begin{equation}\label{ntoQ}
    Q(\ell) = \conc \int_{\rmin}^{\rmax} \noy(\ell, r) \psi(r)\diff r
\end{equation}
where
$$
\conc = \frac{Q(\lmax)}{{\int_{\rmin}^{\rmax} \psi(r)\diff r}}
$$
is the ratio between the number of particles and the number of chords measured  by the sensor, which depends on the experimental conditions.

For a given radius $r$, and a given orientation of the particle, encoded by $(\phi, \theta)$, the chord length is measured according to the situation described in the previous section. That is, the chord length corresponds to a chord length at constant $y$ for an ellipse in the $(x,y)$-plane (of shape determined by $r$, $\phi$, $\theta$ and $\ee$).
Then, $L<\ell$ is achieved if
the horizontal chord has ordinate
$y$ belonging to the set 
\begin{equation}\label{E:set}
(-y_{\max},-y_{\ell})\cup (y_{\ell},y_{\max})
\end{equation}
where $y_{\max}$ is as in \eqref{E:def_ymax} and $y_\ell$ as in \eqref{E:def_yell}.
Since $\ell<\frac{2r}{\sqrt{\alpha}}$ with $\alpha$ as in \eqref{E:alpha}, the probability that $L<\frac{2r}{\sqrt{\alpha}}$ is full. Hence the probability that the measured chord length $L$ is less than $\ell$ is given by
$$
\frac{2(y_{\max}-y_{\ell})}{2y_{\max}}=1- \sqrt{1-\left(\dfrac{\ell}{2r}\right)^2\alpha},
$$
which means that the ordinate of the chord length is chosen uniformly in the set \eqref{E:set}.

Uniformly choosing an orientation of the spheroid
means that the angles $(\phi,\theta)$
are picked in $[0,2\pi]\times [0,\pi]$
according to the probability measure $\diff \mu= \frac{\sin \theta}{4\pi} \diff \phi\diff \theta$.
Then, by the law of total expectation,
\begin{equation}\label{eq:ker}
\noy(\ell,r)
=
1-
\int_{\phi=0}^{2\pi}\int_{\theta=0}^\pi
\sqrt{1-\left(\dfrac{\ell}{2r}\right)^2\alpha_\ee(\phi,\theta)}\frac{\sin \theta}{4\pi} \diff \theta\diff \phi,
\end{equation}
with
\begin{equation}
\alpha_\ee(\phi,\theta)=\frac{\cos ^2\phi}{\cos^2\theta+\ee^2\sin^2 \theta }+\sin ^2\phi.
\end{equation}
Combining the expression of $\noy$ with \eqref{ntoQ}, we get a function that maps a PSD of spheroids to the corresponding cumulative CLD up to the constant $\conc$.
In particular, if $\bar\psi$ is a Dirac distribution at some fixed radius $r$ (which means that all particles have the same radius $r$), then \eqref{nbtoQb} yields $\bar Q(\ell) = \noy(\ell, r)$.
In Figure~\ref{fig:noy}, we plot $\bar Q(\ell)$ for a Dirac distribution of particles at $r = 1$mm, and three different values of $\ee$. This emphasizes the influence of the shape parameter on the CLD.
\begin{figure}[ht!]
    \centering
    \includegraphics[width=.5\linewidth]{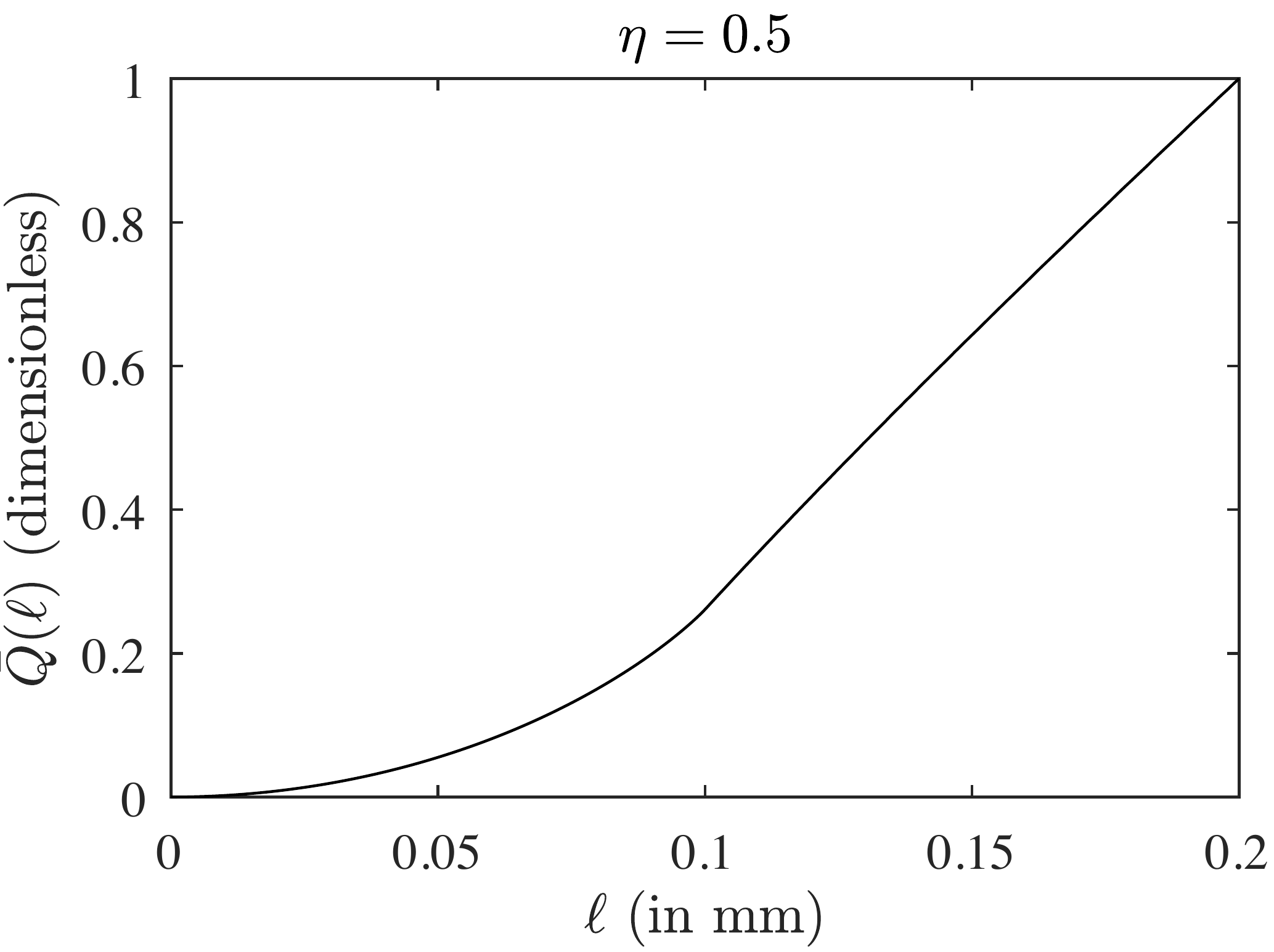}

    \includegraphics[width=.5\linewidth]{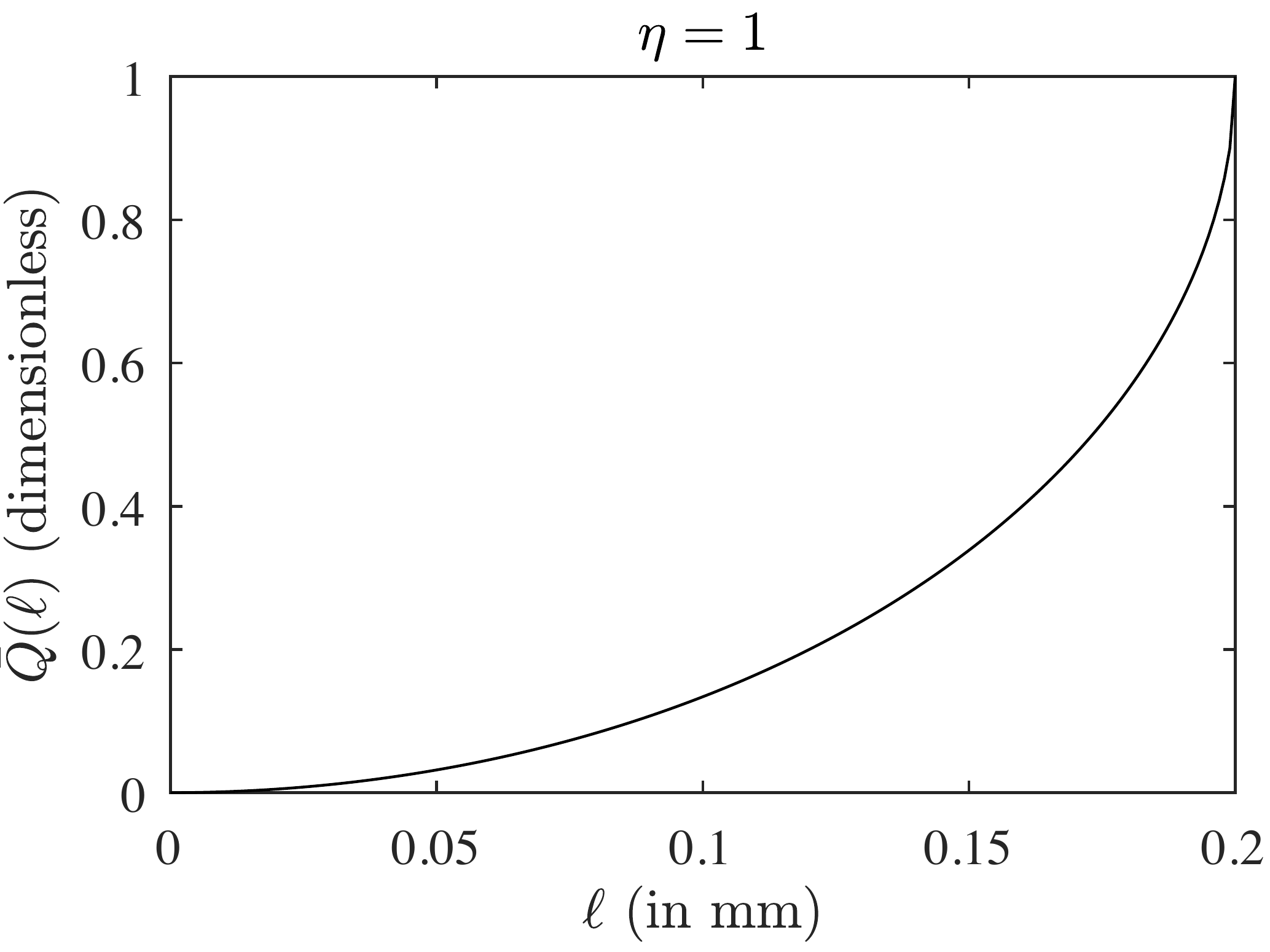}

    \includegraphics[width=.5\linewidth]{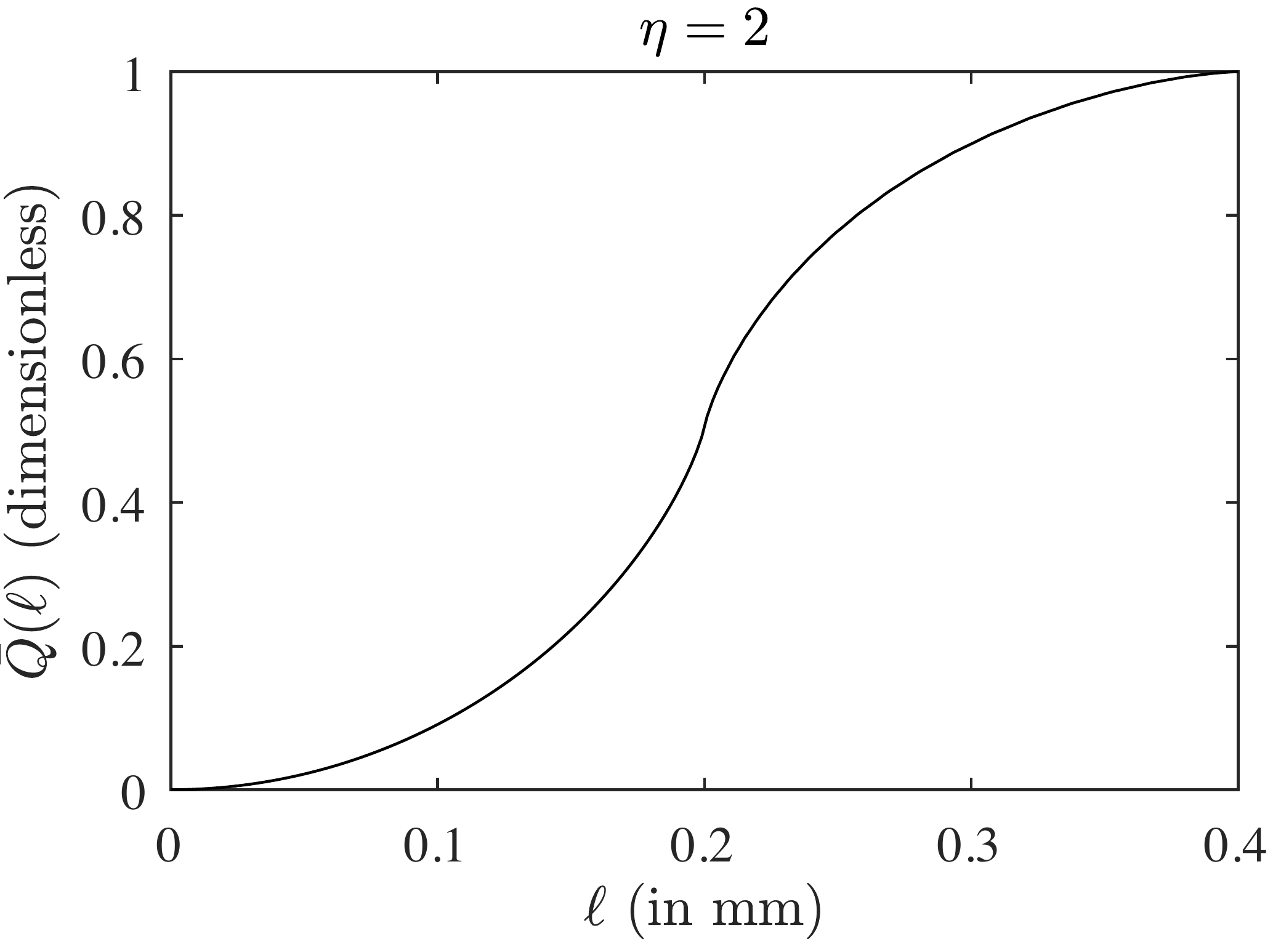}

    \caption{
    Normalized CLD $\bar Q$ associated to a Dirac distribution of spheroids at $r = 1$mm for $\ee=0.5,\, 1,\, 2$.
    }
    \label{fig:noy}
\end{figure}

In the two following sections, we consider the problem of reconstruction of the PSD from the CLD in two different cases of study.

\section{A regularization method for the single-shape case CLD-to-PSD problem}\label{sec3}

Consider a PSD of spheroids sharing the same shape parameter $\ee$.
According to \eqref{ntoQ}, it is possible to compute the corresponding cumulative CLD up to the knowledge of the parameter $\conc$. Conversely, for a given CLD, is it possible to estimate the corresponding PSD?
This question is a crucial issue in process control. Indeed, Process Analytical Technologies (PATs) like the FBRM sensor are able to measure the CLD online, for example during a crystallization process. But the main distribution to be known, and governing the physico-chemical properties of solids, is the PSD.
In this section, we propose a two-steps procedure to recover the PSD from two measures: the CLD, and the solid concentration in the reactor.
\begin{enumerate}
    \item First, using the knowledge of the CLD and \eqref{nbtoQb}, we estimate the renormalized PSD $\bar \psi$.
    \item Second, using the CLD and the solid concentration, we estimate the number of particles per unit of volume $\int_{\rmin}^{\rmax}\psi(r)\diff r$.
\end{enumerate}
Combining these steps with the relation
\begin{equation}
    \psi(r) = \bar \psi(r) \int_{\rmin}^{\rmax}\psi(\rho)\diff \rho, \qquad \forall r\in[\rmin, \rmax],
\end{equation}
we obtain an estimation of the PSD $\psi$.
Sometimes, the knowledge of the number of particles is not to be determined: only the ``shape'' of the PSD is of interest. In this case, only the first step needs to be applied. In numerical simulations, we focus on this first step.

\subsection{Estimation of \texorpdfstring{$\bar \psi$}{p} with a Tikhonov regularization procedure}

Let $X = L^2((\rmin, \rmax); \R)$ be the set of real square integrable functions over $(\rmin, \rmax)$,
and $Y = L^2((0, \lmax); \R)$ with $\lmax = 2\rmax\max(\ee, 1)$.
Then a (renormalized) PSD may be viewed as an element of $X$, while a (renormalized) PSD is an element of $Y$.
Let us define the following map:
\fonction{\mathcal{K}}{X}{Y}{\bar\dtc}{\left(\ell\mapsto
\int_{\rmin}^{\rmax}\noy(\ell, r)\bar\dtc(r)\diff r
\right)}
Equation~\eqref{nbtoQb} may be rewritten as
\begin{equation}\label{pbinv1}
\mathcal{K} \bar \psi = \bar Q.
\end{equation}
For a given CLD $q$, it is easy to compute the cumulative renormalized CLD $\bar Q$.
Then, reconstructing $ \bar \psi$ from $\bar Q$ is solving the inverse problem \eqref{pbinv1} with unknown $ \bar \psi$ in $L^2((\rmin, \rmax); \R)$.
However, this problem admits a solution only if $\bar Q$ lies in the image of $\mathcal{K}$, denoted by $\Im \mathcal{K} = \{\mathcal{K} \bar \psi, \bar \psi \in X\}$.
Due to measurements noise on $\bar Q$, this condition is generally not satisfied. To overcome this problem, we reformulate \eqref{pbinv1} as a minimization problem:
\begin{equation}\label{pbinv2}
\textit{Find } \bar \psi\in X \textit{ minimizing } \|\mathcal{K}  \bar \psi - \bar Q\|^2.
\end{equation}
where $\|\cdot\|$ denotes the $L^2$-norm, that is,
\begin{equation}
    \|\mathcal{K}  \bar \psi - \bar Q\|^2
    = \int_0^{\lmax} |(\mathcal{K}  \bar \psi)(\ell) - \bar Q(\ell)|^2\diff \ell
\end{equation}
Denoting by $\argmin_{ \bar \psi\in X} \|\mathcal{K}  \bar \psi - \bar Q\|^2$ the set of solutions of \eqref{pbinv2}, the following facts hold (see, \emph{e.g.}, \cite{Tikhonov}):
\begin{itemize}
    \item If $\mathcal{K}$ is injective, then \eqref{pbinv2} has at most one solution.
    \item If $\bar Q \in \Im\mathcal{K} \oplus \left(\Im\mathcal{K}\right)^\perp$, then the set $\argmin_{ \bar \psi \in X}\|\mathcal{K}  \bar \psi - \bar Q\|^2$ is closed, convex and non-empty (in particular \eqref{pbinv2} admits at least one solution).
    \item If $\mathcal{K}$ is injective and admits a left inverse denoted by $\mathcal{K}^{-1}$, then the unique solution of \eqref{pbinv2} is $ \bar \psi=\mathcal{K}^{-1}\bar Q$.
\end{itemize}

The space $\Im\mathcal{K} \oplus \left(\Im\mathcal{K}\right)^\perp$ being dense in $Y$, we assume in the following that $\bar Q$ lies in this set.
In section~\ref{sec:inj}, we prove the following proposition.
\begin{restatable}{theorem}{propinj}
The operator $\mathcal K$ is injective.
\end{restatable}

Therefore, the problem \eqref{pbinv2} admits exactly one solution.
However, numerically computing this solution remains challenging, because the problem is still ill-posed.
Indeed, the operator $\mathcal{K}$ is compact, as an integral operator with square-integrable kernel.
Hence, its left-inverse can not be continuous, which implies that any small measurement noise on $\bar Q$ leads to a major perturbation of the estimated renormalized PSD $ \bar \psi$.
To tackle this issue, a typical approach is the Tikhonov regularization procedure.

\begin{proposition}[see, \emph{e.g.}, \cite{Tikhonov}]
For any $\delta>0$,
the minimization problem 
\begin{equation}\label{pbinv3}
\text{Find } \bar \psi\in X \text{ minimizing } \|\mathcal{K}  \bar \psi - \bar Q\|^2 + \delta \| \bar \psi\|^2.
\end{equation}
admits a unique solution, which depends continuously on $\bar Q$.
\end{proposition}

The Tikhonov regularization consists in replacing the ill-posed problem \eqref{pbinv2} by the well-posed \eqref{pbinv3}.
The parameter $\delta$ is called the \emph{regularization parameter}.
Letting $\delta$ tend towards zero, we recover the original problem \eqref{pbinv2}. As $\delta$ tends towards infinity, the solution of \eqref{pbinv3} tends towards zero.
The choice of $\delta$ is a trade-off: the regularized problem must be sufficiently close to the original problem ($\delta$ sufficiently small) to have a similar solution, but not too close to remain robust to measurement noise ($\delta$ sufficiently large). It must be experimentally selected.
One can interpret $\delta$ as a  confidence measure: the more uncertain the sensor is, the larger $\delta$ should be.
Finally, since $ \bar \psi$ is known to be a probability density function, one can constrain the minimization problem:
\begin{equation}\label{pbinv4}
\textit{Find } \bar \psi\in X \textit{ minimizing } \|\mathcal{K}  \bar \psi - \bar Q\|^2 + \delta \| \bar \psi\|^2
\textit{ subject to }  \bar \psi \geq 0.
\end{equation}
Denoting by $ \bar \psi$ the solution of this latter problem, we now aim to find the PSD $\psi$.

\subsection{Estimation of the number of particles with solid concentration}

In this section, we propose to estimate $\int_{\rmin}^{\rmax}\psi(r)\diff r$
by using a measurement of the solid concentration $C_s$ (in $\text{kg}$ of solid per $\text{kg}$ of solvent).
Let $\rho_s$ (in $\text{kg}.\text{m}^{-3}$) be the density of the solute in solid phase, $M_e$ be the solvent mass (in $\text{kg}$), and $V_s$ (in $\text{m}^{3}$) be the volume occupied by the particles in the reactor. Then
$C_s = \frac{\rho_s}{M_e}V_s$ and
\begin{equation}
    V_s = \frac{4\pi}{3}\ee\int_{\rmin}^{\rmax}\psi(r)r^3\diff r
    = \frac{4\pi}{3}\ee\int_{\rmin}^{\rmax}\bar\psi(r)r^3\diff r
    \int_{\rmin}^{\rmax}\psi(r)\diff r.
\end{equation}
Using the estimation of $\bar\psi$ obtained in the previous step,
we get
\begin{equation}
    \int_{\rmin}^{\rmax}\psi(r)\diff r
    =\frac{3}{4\pi\ee} \frac{M_e}{\rho_s}\frac{C_s}{\int_{\rmin}^{\rmax}\bar\psi(r)r^3\diff r}.
\end{equation}
Thus, if $\rho_s$, $M_e$ and $\ee$ are known, and $\bar\psi$ is estimated in the previous step, it is possible to estimate the number of particles per unit of volume with a measurement of the solid concentration.

\subsection{Numerical simulations}

For simulations, we consider a bi-modal normalized PSD $\bar\psi(r)$ of spheroid particles with shape parameter $\ee = 2$ and radius $r$ between $\rmin = 1.0\times 10^{-4}$m and $\rmax = 3.0\times 10^{-4}$m,
attaining its maximum at $r = 1.5\times 10^{-4}$m and $r = 2.5\times 10^{-4}$m.
More precisely, we choose
\begin{equation}\label{eq:psdexp}
\bar\psi(r) = \frac{e^{-30(r-1.5\times 10^{-4})^2}+e^{-30(r-2.5\times 10^{-4})^2}}{\int_{1\times 10^{-4}}^{3\times 10^{-4}} e^{-30(\rho-1.5\times 10^{-4})^2}+e^{-30(\rho-2.5\times 10^{-4})^2} \diff \rho}.
\end{equation}

The corresponding CLD $\bar q$ satisfies $\bar Q = \opc \bar\psi$, where $\bar Q$ is the cumulative CLD. The chord lengths $\ell$ lie in $[0, \lmax]$, with $\lmax = 2\rmax\ee = 12$mm.
We add a zero mean Gaussian noise to $q$ with variance deviation of $2\%$ of the maximum of $q$.
Then, we apply the Tikhonov regularization procedure to estimate $\psi$ from the noised CLD $q$.
Intervals $[\rmin, \rmax]$ and $[0, \lmax]$ are discretized with $200$ equally spaced points.
We use three different values of the regularization parameter $\delta(=10^{-5},\, 10^{-3},\,10^{-1})$.
We plot the results in Figure~\ref{fig:tikho}.
\begin{figure}[ht!]
    \centering
    \includegraphics[width=.5\linewidth]{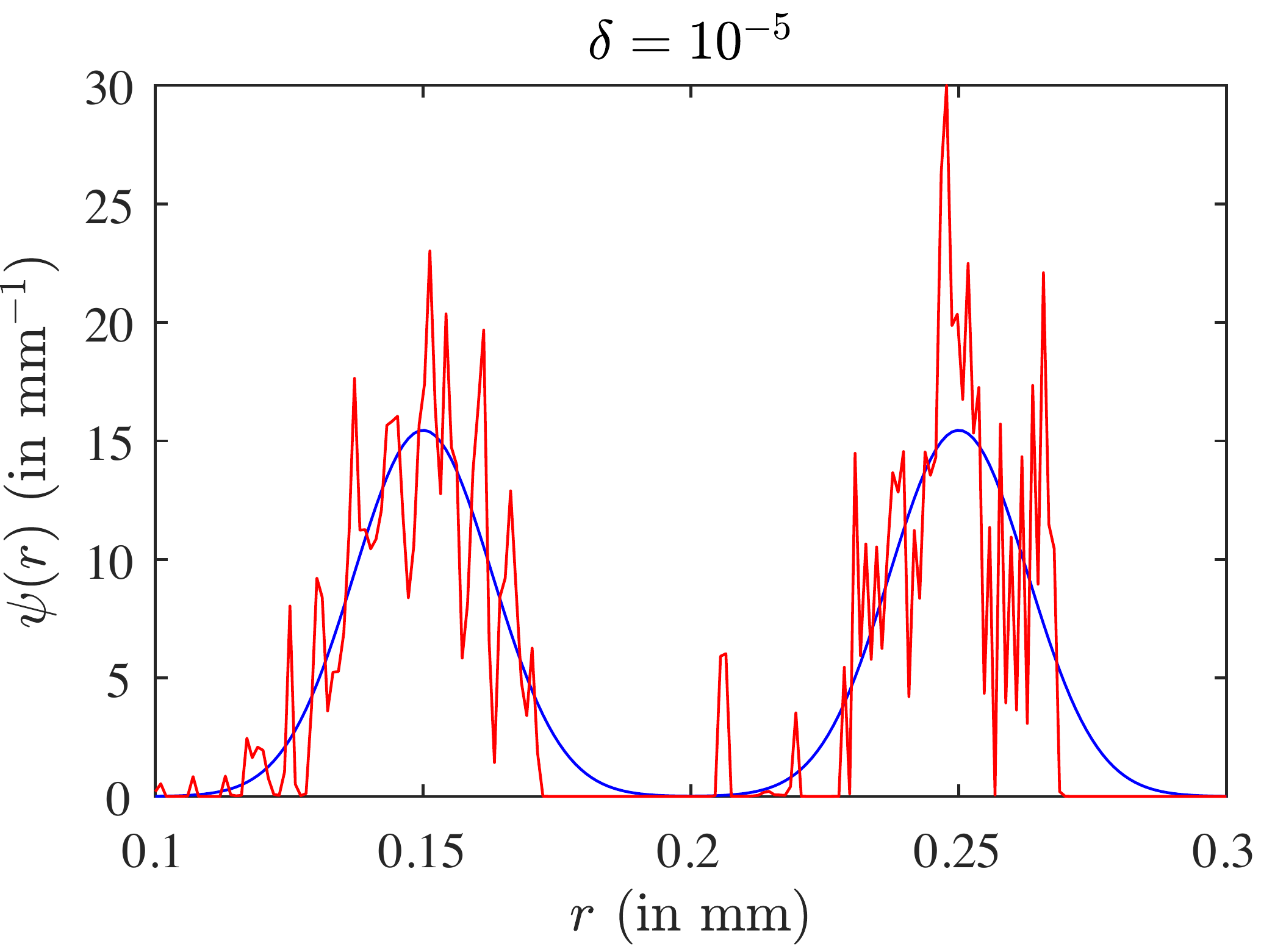}\\
    \includegraphics[width=.5\linewidth]{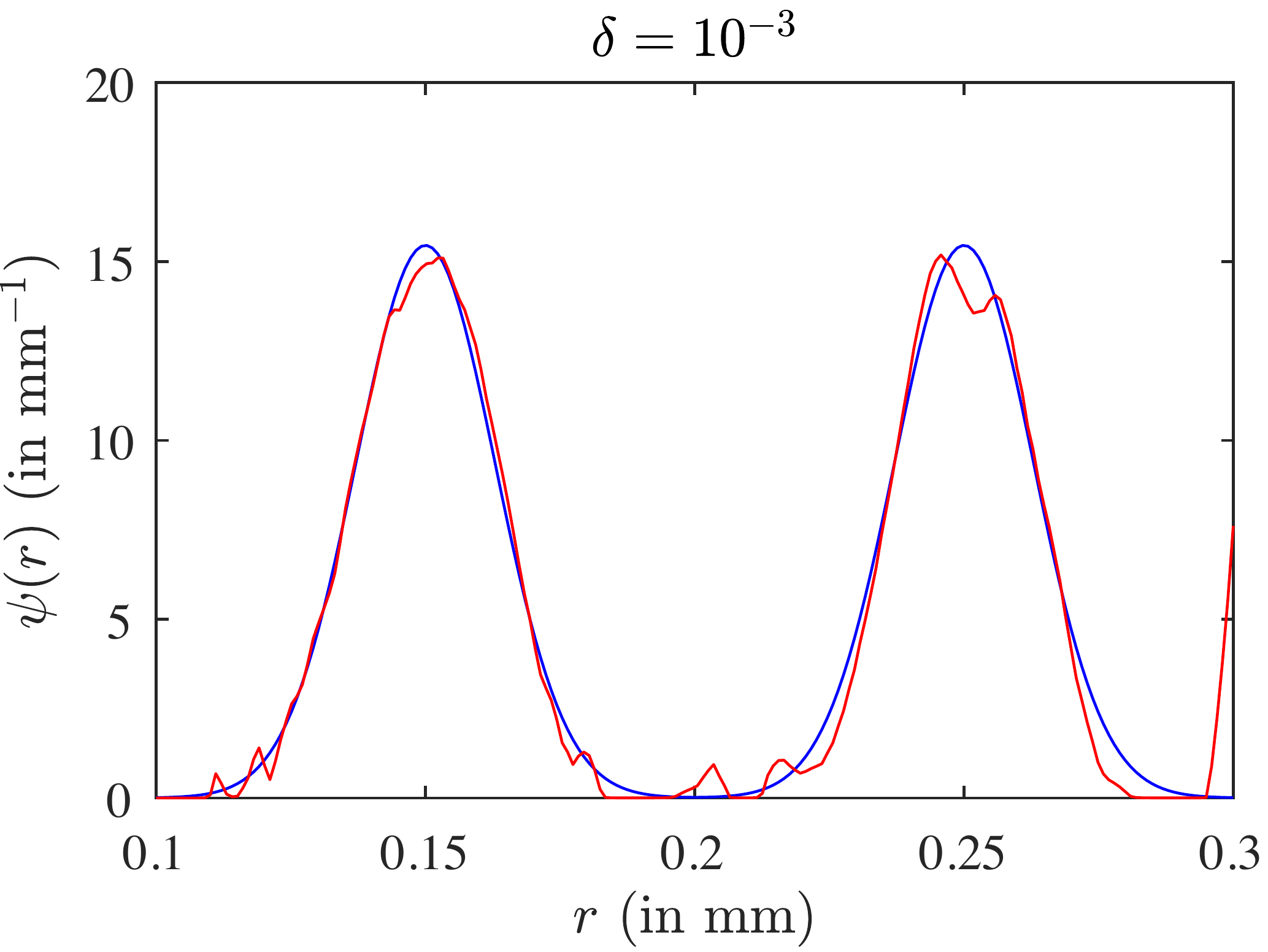}\\
    \includegraphics[width=.5\linewidth]{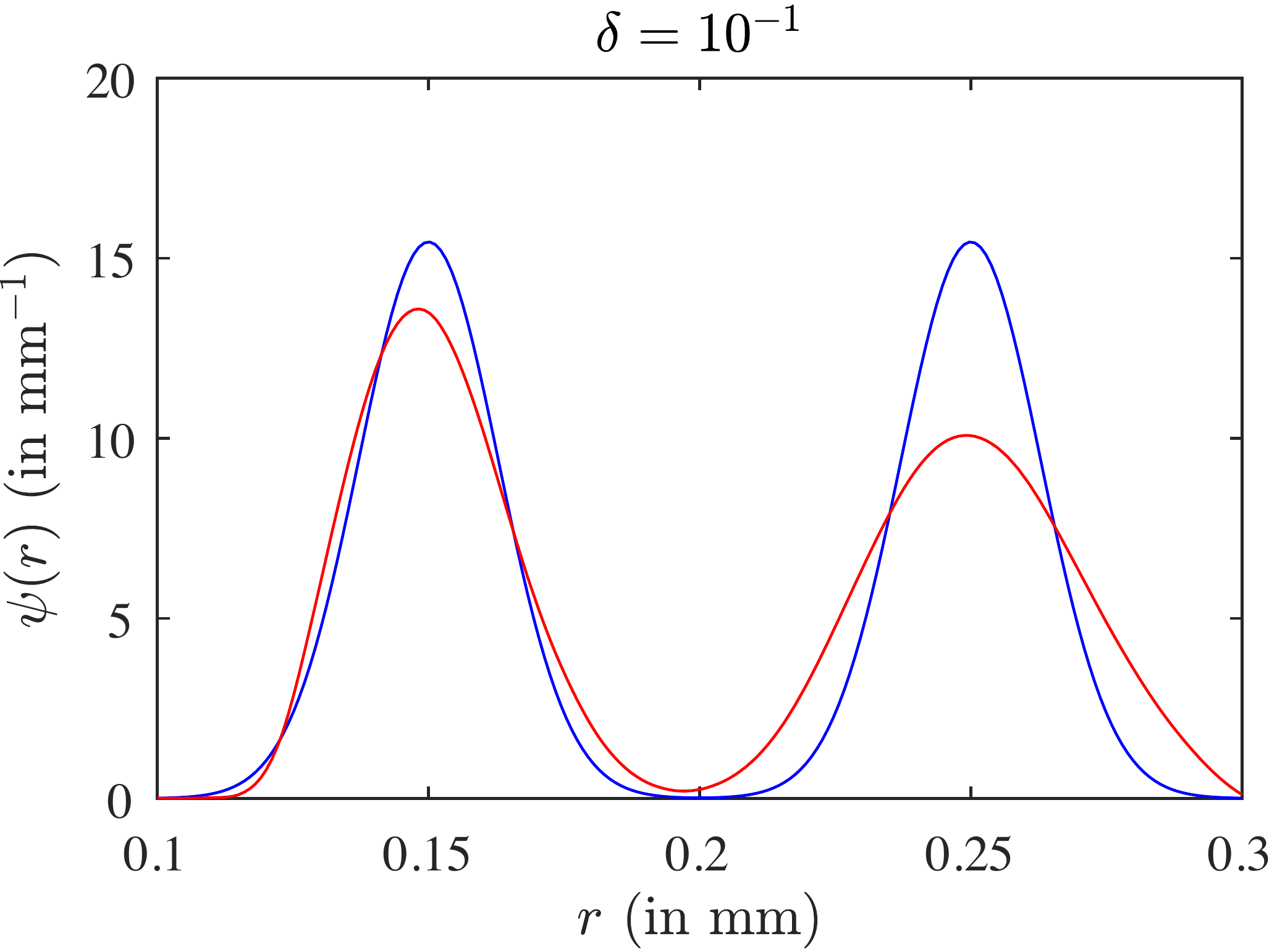}
    \caption{
    Estimation of the PSD by the Tikhonov regularization method.
    In blue: the PSD given by \eqref{eq:psdexp}. In red: the PSD estimated by the Tikhonov regularization method for $\delta = 10^{-5},\, 10^{-3},\, 10^{-1}$.
    }
    \label{fig:tikho}
\end{figure}
For all the considered values of $\delta$, the bi-modality of the PSD is recovered by the estimation.
However, when $\delta = 10^{-5}$, the regularization parameter is too small. The discontinuity issues of the  non-regularized problem \eqref{pbinv1} still appear.
% The minimization problem is not sufficiently regular: \eqref{pbinv2} is close of the non-regularized problem \eqref{pbinv1}.
On the contrary, $\delta = 10^{-1}$ is too large. The regularized problem is too far from the original minimization problem and some information on the amplitude of the PSD is lost.
With $\delta = 10^{-3}$, we recover a satisfying estimation of the original PSD by balancing these two effects.

\subsection{Injectivity analysis}\label{sec:inj}

Let us recall and prove the following statement.
\propinj*

\begin{proof}

Let $\dtc\in L^2((\rmin, \rmax); \R)$ such that $\mathcal{K}\dtc = 0$. Then, for almost every $\ell\in(0, \lmax)$, we have:
\begin{equation}
\begin{aligned}
0&=\int_{\rmin}^{\rmax} \noy(\ell, r)\dtc(r)\diff r
\\
&=
\int_{\rmin}^{\rmax}\dtc(r)\diff r-
\int_{\rmin}^{\rmax}
\dtc(r)
\int_{\phi=0}^{2\pi}\int_{\theta=0}^\pi
\sqrt{1-\left(\dfrac{\ell}{2r}\right)^2\alpha_\ee(\phi,\theta)}\frac{\sin \theta}{4\pi} \diff \theta\diff \phi\diff r.
\end{aligned}
\end{equation}
Let us consider the sequence $(\mathcal{K}\dtc)^{(2n)}(0)$. It can be computed using differentiation of the parameter integral.
The function $\ell \mapsto \sqrt{1-\left(\dfrac{\ell}{2r}\right)^2\alpha}$ is analytic at $0$, and we have (from the series expansion of $\sqrt{1-\ell^2}$) that 
$$
\left.\frac{\diff^{2n}}{\diff \ell^{2n}}\sqrt{1-\left(\dfrac{\ell}{2r}\right)^2\alpha}\right|_{\ell=0}
=
\frac{(2n)!}{(n!)^2(1-2n)4^{2n}}\frac{\alpha^{n}}{r^{2n}}
$$
Hence, for $n\geq 1$, 
$$
(\mathcal{K}\dtc)^{(2n)}(0)
=
\frac{(2n)!}{(n!)^2(1-2n)4^{2n}}
\int_{\phi=0}^{2\pi}\int_{\theta=0}^\pi
\alpha_\ee^{n}(\phi,\theta) \frac{\sin \theta}{4\pi} \diff \theta\diff \phi
\int_{\rmin}^{\rmax}
\frac{\dtc(r)}{r^{2n}}
\diff r.
$$
Let us denote
$$
a_n(\ee)=
\int_{\phi=0}^{2\pi}\int_{\theta=0}^\pi
\alpha_\ee^{n}(\phi,\theta) \frac{\sin \theta}{4\pi} \diff \theta\diff \phi,\qquad
b_n = \frac{(2n)!}{(n!)^2(1-2n)4^{2n}}
$$
so that 
$$
(\mathcal{K}\dtc)^{(2n)}(0)=a_n(\ee)b_n
\int_{\rmin}^{\rmax}
\frac{\dtc(r)}{r^{2n}}\diff r.
$$
Since $\mathcal{K}\dtc$ is constantly equal to $0$,
$\mathcal{K}\dtc^{(2n)}(0)=0$ for all $n\in \N^*$.
Since $a_n(\ee)b_n>0$ for all $\ee>0$ and all $n\in \N^*$, having $(\mathcal{K}\dtc)^{(2n)}(0) =0$ for all $n\in \N^*$ implies that 
$$
\int_{\rmin}^{\rmax}
\frac{\dtc(r)}{r^{2n}}\diff r
=0 \qquad \forall n\in \N^*.
$$
The family $\left(r\mapsto\frac{1}{r^{2n}}\right)_{n\in \N^*} $ is total in $L^2((\rmin, \rmax); \R)$ since $\rmin>0$ (see \cite[Section 6]{brivadis:hal-02529820}), which implies that $\dtc(r)=0$ for all $r\in[\rmin, \rmax]$. 
\end{proof}

\section{A dynamical observer for the multi-shape CLD-to-PSD problem}\label{sec4}

In this section, we consider that spheroids of different shape factors $\ee_i$, $i\in\{1,\dots,N\}$ are in the reactor at the same time.
The PSD associated to each shape is denoted by $\psi_i$.
This situation frequently occurs in batch crystallization processes: crystals with different shapes appear during the process due to polymorphism.
Then, the CLD data collected by the sensor is the sum of the CLDs associated to each PSD. More precisely, with the notations of \eqref{ntoQ}, the measured cumulative CLD $Q$ satisfies
\begin{equation}
    Q(\ell) = \sum_{i=1}^N \conc_i \int_{\rmin}^{\rmax} \noy_i(\ell, r) \psi_i(r)\diff r,
\end{equation}
where $k_i$ is the kernel defined in \eqref{eq:ker} with $\ee = \ee_i$ and
$$
\conc_i = \frac{{\int_{0}^{\lmax} q(\ell)\diff \ell}}{{\int_{\rmin}^{\rmax} \psi_i(r)\diff r}}.
$$

Then, a natural question to ask is: is it possible to reconstruct $\conc_i \psi_i$ from the knowledge of $Q$, as it is done in the case $N=1$ in the previous section?
% Unfortunately, when $N>1$, the injectivity of the operator
Unfortunately, when $N>1$, the operator
\fonction{\mathcal{K}}{X^N}{Y}{(\bar\dtc_i)_{1\leq i\leq N}}{\left(\ell\mapsto
\sum_{i=1}^N \int_{\rmin}^{\rmax}\noy_i(\ell, r)\bar\dtc_i(r)\diff r
\right)}
may not be injective. Indeed, the different PSDs are intertwined in the CLD. In particular, in the case where $\ee_i = \ee_j$ for some $i\neq j$, there is no way to differentiate the part of the CLD due to $\psi_i$ and the part due to $\psi_j$.
% remains an open and difficult problem which we have not been able to solve.
Therefore, applying the Tikhonov regularization procedure in this case is not a convenient approach.
\medskip

However, in the case of a crystallization process, there is one more information that we can use to reconstruct PSD from CLD: a model of the PSD dynamics, based on a population balance equation.
Hence, the goal of this section is to estimate the PSD using a dynamical model of crystallization process, and a measure of the CLD over a finite time interval.
To do so, we will use a new approach based on the BFN algorithm.
Let us first determine the population balance equation.

\subsection{Population balance}

We consider an elementary model of a batch crystallization process occurring on a time window $[0, \tmax]$ (in $\text{s}$) (see \emph{e.g.}, \cite{Mullin, Mersmann}).
Polymorphism is a common phenomenon that may occur during crystallization: crystals may have several metastable shapes.
We assume that only a finite number $N$ of shapes may appear during the process, and that each of these shapes can be modeled as an ellipsoid.
As in section~\ref{sec:ellips}, an ellipsoid is fully characterized by a radius $r$ (in $\text{m}$) and an adimensional shape parameter $\ee$.
To each shape $i\in\{1,\dots,N\}$, we associate a parameter $\ee_i$.
We denote by $\dtc_i(t, \cdot)$ (in $\text{m}^{-1}.\text{m}^{-3}$) the PSD of particles having the shape $i$ at time $t$ in the reactor,
so that $\int_{r_1}^{r_2}\dtc_i(t, r)\diff x$ (in $\text{m}^{-3}$) is the number of crystals in the reactor at time $t$ having the shape $i$ and a radius $r$ between $r_1$ and $r_2$.
Let $\rmax$ be a maximal radius that no crystals of any shape can reach during the process (such as the size of the reactor):
\begin{equation}\label{eq:rmax}
    \dtc_i(t, \rmax) = 0,\qquad \forall t\in[0, \tmax],\ \forall i\in\{1,\dots,N\}.
\end{equation}
We assume that all crystals of all shapes appear at the same minimal radius $\rmin$, and denote by $u_i(t)$ (in $\text{m}^{-1}.\text{m}^{-3}$) the appearance of particles of size $\rmin$ and shape $i$ at time $t$:
\begin{equation}
    \dtc_i(t, \rmin) = u_i(t),\qquad \forall t\in[0, \tmax].
\end{equation}
The function $u_i$ is linked to the nucleation rate $R_i$ (in $\text{m}^{-3}.\text{s}^{-1}$) and the growth rate $G_i$ (in $\text{m}.\text{s}^{-1}$) of crystals of shape $i$ by the following formula:
\begin{equation}
    u_i(t) = \frac{R_i(t)}{G_i(t)},\qquad \forall i\in\{1,\dots,N\}.
\end{equation}
The growth rate is supposed to be positive at any time.
Considering McCabe hypothesis, $G_i$ is independent of the crystals size (but  depends on the shape).
Assuming that the different shapes do not interact with each other, the population balance leads to
\begin{equation}
    \frac{\partial \dtc_i}{\partial t}(t, r) + \vit_i(t) \frac{\partial \dtc_i}{\partial r}(t, r) = 0,\qquad \forall i\in\{1,\dots,N\}.
\end{equation}
Finally, assume that seed particles with PSD $\dtc_{i, 0}$ for each shape $i$ may lie in the reactor at time $t=0$:
\begin{equation}
    \dtc_i(0, r) = \dtc_{i, 0}(r),\qquad \forall r\in[\rmin, \rmax].
\end{equation}

To summarize, the evolution of the PSD through the process follows the set of partial differential equations (PDEs)
\begin{equation}
\forall i\in\{1,\dots,N\},\
\begin{aligned}
\begin{cases}
\frac{\partial \dtc_i}{\partial t}(t, r) + \vit_i(t) \frac{\partial \dtc_i}{\partial r}(t, r) = 0
& \forall t\in(0, \tmax), \forall r\in(\rmin, \rmax)\\
\dtc_i(0, r) = \dtc_{0, i}(r)
& \forall r\in[\rmin, \rmax]\\
\dtc_i(t, \rmin) = \cont_i(t)
&\forall t\in[0, \tmax]
\end{cases}
\end{aligned}
\label{systbilan}
\end{equation}
with the additional boundary conditions \eqref{eq:rmax}.

% Let $r_0 = \rmin - \max_{1\leq i\leq N} \int_0^{\tmax}G_i(t)\diff t$ and $r_1 = \rmax$.
Since $u_i$ is not supposed to be measured, it is part of the unknown data to be reconstructed. Define $\psi_i(t, r)$ for $\rmin - \int_t^{t+\tmax} G_i(s)\diff s \leq r\leq \rmin$ in the following manner:
\begin{equation}
    \psi_i(t, r) = u_i(t+\tau) \text{ with } \tau\geq0 \text{ such that } \int_t^{t+\tau} G_i(s)\diff s = \rmin-r.
\end{equation}
Roughly speaking, $\psi_i(t, r)$ for $r<\rmin$ represents crystals that did not yet appear at time $t$, but will appear later at some time $t+\tau$.
If $t+\tau>\tmax$, set $\psi_i(t, r) = 0$.
Combining all the PSDs in a unique vector $\psi = (\psi_i)_{1\leq i\leq N}$,
$G(t) = \diag((G_i(t))_{1\leq i\leq N})$
and
$\psi_0(r) = (\psi_{0,i}(r))_{1\leq i\leq N}$,
system \eqref{systbilan} can be rewritten as
\begin{equation}
\begin{aligned}
\begin{cases}
\frac{\partial \dtc}{\partial t}(t, r) + \vit(t) \frac{\partial \dtc}{\partial r}(t, r) = 0
& \forall t\in(0, \tmax), \forall r\in(r_0, r_1)\\
\dtc(0, r) = \dtc_0(r)
& \forall r\in[r_0, r_1]
\end{cases}
\end{aligned}
\label{syst}
\end{equation}
where $r_0 = \rmin - \max_{1\leq i\leq N}\int_0^{\tmax} G_i(s)\diff s$ and $r_1 = \rmax$
and
with periodic boundary conditions $\psi(t, \rmin) = \psi(t, \rmax)$ (since the right boundary term does not influence $\psi(t, r)$ for $r>\rmin$ and $t\leq \tmax$).
Then, any solution $\psi$ of \eqref{syst} is such that $\psi(t, r)$ is the corresponding solution of \eqref{systbilan} when restricted to $t\in[0, \tmax]$ and $r\in[\rmin, \rmax]$.

\begin{proposition}[Well-posedness]\label{well}
If $G_i$ is positive and $C^1$, $\psi_{0, i}\in L^2((\rmin, \rmax); \R)$ and $u_i\in L^2((0, \tmax); \R)$ for all $i\in\{1,\dots,N\}$,
then system \eqref{syst} admits a unique solution $\psi\in C^0((0, \tmax); L^2((r_0, r_1); \R)^N)$.
\end{proposition}
\begin{proof}
The proof relies on the theory of linear evolution systems (see \emph{e.g.} \cite{Pazy}).
Let $X = L^2((r_0, r_1); \R)$ and $\mathcal{D} = \{\psi\in X \mid \psi' \in X, \psi(r_0) = \psi(r_1) \}$.
The operator $-G(t)\frac{\partial}{\partial r}: \mathcal{D}^N\to X^N$ is linear, unbounded, and skew-adjoint for all $t\in[0, \tmax]$. Since $G$ is $C^1$, $t\mapsto -G(t)\frac{\partial\psi}{\partial r}$ is continuously differentiable for all $\psi\in\mathcal{D}^N$.
Hence, according to \cite[Chapter 5, Theorem 4.8]{Pazy}, it is the generator of a bidirectional evolution system on $X^N$.
In particular, \eqref{syst} admits a unique solution $\psi\in C^0((0, \tmax); X^N)$ for each $\psi_0\in X^N$.
\end{proof}

\subsection{Back and Forth Nudging algorithm}

Abusing notations, let us replace $\conc_i\psi_i$ by $\psi_i$, which satisfies the same PDE \eqref{systbilan}.
Suppose that $\conc_i$ is independent of time for all $i$, that is, the ratio between the number of particles and the number of chords seen by the sensor is constant.
Our goal is to estimate
$\psi(t, r) = (\psi_i(t, r))_{1\leq i\leq N}$ from
the knowledge of the cumulative CLD $Q(t, \ell)$ over the time interval $[0, \tmax]$, given by
\begin{equation}
    Q(t, \ell) = \sum_{i=1}^N \int_{\rmin}^{\rmax} \noy_i(\ell, r) \psi_i(t, r)\diff r.
\end{equation}
Let $X = L^2((r_0, r_1); \R)$, $\lmax = 2\rmax\max_{1\leq i\leq N}(\ee_i)$ and $Y = L^2((0, \lmax); \R)$, so that $Q(t, \cdot)\in Y$ for all $t\in[0, \tmax]$.
Define the operator
\fonction{\opc}{X^N}{Y}{\psi}{\left(\ell\mapsto
\sum_{i=1}^N \int_{\rmin}^{\rmax} \noy_i(\ell, r)\psi_i(r)\diff r\right).}
Its adjoint operator is
\fonction{\opc^*}{Y}{X^N}{Q}{\left(r\mapsto
\int_{0}^{\lmax} \noy_i(\ell, r)Q(\ell)\diff \ell\right)_{1\leq i\leq N}}
with $k_i(\ell, r) =0$ for $r\notin [\rmin, \rmax]$ or $\ell\notin[0, 2\rmax\ee_i]$.

Since the operator $\opc$ may not be injective, trying to invert it to recover $\psi$ from $Q$ at each time $t$ is not a suitable approach, and we rather make use of the dynamics \eqref{syst}.
To do so, we apply the so-called BFN algorithm.
It is an iterative method based on forward and backward dynamical observers.
Usually, observers are used in control engineering to estimate the state of a system online by using the measure. 
They are designed to converge towards the real state of the system over an infinite time interval.
On the contrary, BFN algorithm makes use of observers forward and backward on a finite time window. Each observer using the estimation made by the previous one, they are supposed to converge iteratively to the real state of the system.
More precisely, we use the BFN algorithm proposed in \cite{brivadis:hal-02529820}, based on Luenberger forward and backward observers.
In our context, the observer system is the following, where $\mu>0$ is a degree of freedom:
\begin{align}
&\begin{cases}
\frac{\partial \etath^{2n}}{\partial t}(t, r) = -\vit(t) \frac{\partial \etath^{2n}}{\partial r}(t, r)- \mu\opc^*(\opc\etath^{2n}(t, \cdot)-\bar Q(t, \cdot))
\quad
\forall t\in(0, \tmax), \forall r\in(r_0, r_1)\\
\etath^{2n}(0, r) =
\begin{cases}
\etath^{2n-1}(0, r) & \text{if } n\geq 1\\
\etath_0(r) & \text{otherwise}
\end{cases}
\quad \forall r\in(r_0, r_1)
\end{cases}
\label{obs2}\\
&\begin{cases}
\frac{\partial \etath^{2n+1}}{\partial t}(t, r) = -\vit(t) \frac{\partial \etath^{2n+1}}{\partial r}(t, r) +\mu\opc^*(\opc\etath^{2n+1}(t, \cdot)-\bar Q(t, \cdot))
% \quad
\
\forall t\in(0, \tmax), \forall r\in(r_0, r_1)\\
\etath^{2n+1}(\tmax, r) = \etath^{2n}(\tmax, r)\quad \forall r\in(r_0, r_1)
\end{cases}
\label{obs2b}
\end{align}

In this system, $\hat\psi^n(t, r)$ is the estimation of the actual PSD $\psi(t, r)$ obtained after $n$ iterations of the algorithm. Note that the algorithm relies only on the knowledge of the normalized CLD $\bar Q(t, \ell)$ on the time interval $[0, \tmax]$.
The following theorem ensures the convergence of $\hat\psi^n$ to $\psi$.
\begin{theorem}\label{th:bfn}
Assume that for all $\psi_0\in X$, the following implication is satisfied:
\begin{equation}\label{cond:obs}
    \left(\forall t\in[0, \tmax],\ \opc \psi(t, \cdot) = 0\right)
    \quad\implies\quad
    \psi_0 = 0,
\end{equation}
where $\psi$ denotes the solution of \eqref{syst} with initial condition $\psi_0$.
Then, for all $\mu>0$, all $t\in[0, \tmax]$ and almost all $r\in[r_0, r_1]$,
\begin{equation}
    \etath^n(t, r)
    \cvl{n\cv+\infty}\psi(t, r).
\end{equation}
% $\etath^n(t, \cdot) - \psi(t, \cdot)$ tends towards $0$ in the weak topology on $X$ as $n$ goes to infinity, that is,
% \begin{equation}
% \langle\etath^n(t) - \psi(t), \varphi\rangle_{X^N} \cvl{n\cv+\infty}0,\qquad \forall\varphi\in X^N.
% \end{equation}
% \begin{equation}
% \forall\varphi\in X^N,\qquad
% \sum_{i=1}^N
% \int_{r_0}^{r_1}\left(\etath_i^n(t, r) - \psi_i(t, r)\right)\varphi_i(r)\diff r \cvl{n\cv+\infty}0.
% \end{equation}
\end{theorem}
\begin{proof}
This result is an application of \cite[Theorem 3.9.(i)]{brivadis:hal-02529820}.
Let $X = L^2((r_0, r_1); \R)$ and $\mathcal{D} = \{\psi\in X \mid \psi' \in X, \psi(r_0) = \psi(r_1) \}$.
As in Proposition~\ref{well}, $-G(t)\frac{\partial}{\partial r}: \mathcal{D}^N\to X^N$ is skew-adjoint for all $t\in[0, \tmax]$.
Moreover, \eqref{cond:obs} states that \eqref{syst} with output $\mathcal{K}\psi$ is observable, that is, its observable subspace is $X$.
Hence, all the hypotheses of \cite[Theorem 3.9.(i)]{brivadis:hal-02529820} are satisfied, so that the BFN algorithm converges to the actual state of the system as the number of iterations goes to infinity.
\end{proof}

Condition \eqref{cond:obs} is an \emph{observability} condition, and can be reformulated in the following way.
If two initial conditions $\psi_0$ and $\tilde\psi_0$
(\emph{i.e.}, $(u_i)$, $(\psi_{0,i})$, $(\tilde u_i)$, $(\tilde\psi_{0,i})$, $1\leq i\leq N$)
are such that
the corresponding cumulative CLDs $Q$ and $\tilde Q$
are the same on the whole time interval $[0, \tmax]$,
then $\psi_0 = \tilde\psi_0$, which implies that the two PSDs are also the same on $[0, \tmax]$.
Indeed, by taking the difference $\psi-\tilde\psi$, we recover \eqref{cond:obs}.
Hence, the main question to investigate is now: when does the observability condition \eqref{cond:obs} holds?
% In this paper, we prose two different contexts in which \eqref{cond:obs} is satisfied. The proof of observability in each context is postponed in appendix.
% In each case, we consider sufficiently regular initial conditions $\psi_0\in H^2(r_0, r_1)$, satisfying the boundary condition \eqref{eq:rmax}.
In some crystallization processes, there are two shapes of crystals of the same species appearing simultaneously in the reactor due to polymorphism. Frequently, one of these shapes is almost spherical, and the other is very elongated (see Figure~\ref{fig:petitspoiscarottes} and the experiments of \cite{GAO} for example).
One of the main results of the paper is that in this case, the system is observable.
Hence, according to Theorem~\ref{th:bfn}, the BFN algorithm is able to estimate the actual PSD of each shape from the knowledge of the CLD during the process.
The proof of the result is postponed in appendix~\ref{app}.

\begin{figure}[ht!]
    \centering
 
    \includegraphics[width=.5\linewidth]{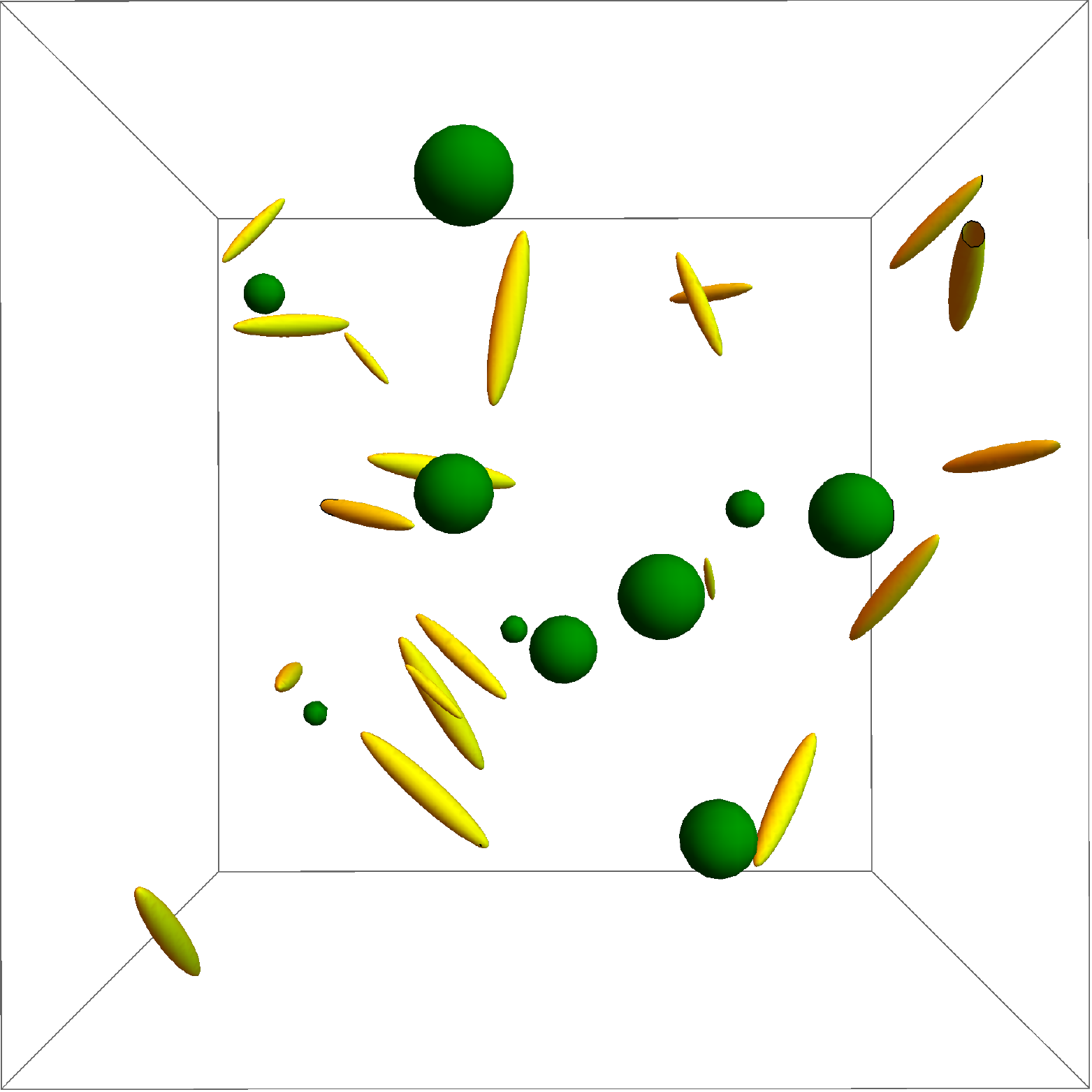}   
    \caption{Simulated suspension of ideal particles of two shapes, distributed in size: spheres ($\eta = 1$, in green) and prolate spheroids ($\eta = 6$, in yellow), in a cubic volume of 1cm$^3$.}
    \label{fig:petitspoiscarottes}

\end{figure}

\begin{restatable}{theorem}{thobs}\label{th:obs}
Consider two clusters of crystals $(N=2)$ with shapes $\ee_1=1$ and $\ee_2>1$.
Assume that their growth rate have constant ratio $\frac{g_1}{g_2}$, \textit{i.e.}, $g_2G_1(t) = g_1G_2(t)$ for all $t\in[0, \tmax]$.
Then for all $\psi_0\in H^2(r_0, r_1)$ satisfying the boundary condition \eqref{eq:rmax},
\begin{equation}\label{eq:obs}
    \left(\forall t\in[0, \tmax],\ \opc \psi(t, \cdot) = 0\right)
    \quad\implies\quad
    \psi_0 = 0,
\end{equation}
\end{restatable}
% \begin{theorem}
% Consider two clusters of crystals $(N=2)$ with shapes $\ee_1=1$ and $\ee_2>1$.
% Assume that their growth rate have constant ratio $\gamma$, \textit{i.e.}, $G_1(t) = \gamma G_2(t)$ for all $t\in[0, \tmax]$.
% Then for all $\psi_0\in H^2(r_0, r_1)$ satisfying the boundary condition \eqref{eq:rmax},
% \begin{equation}
%     \left(\forall t\in[0, \tmax],\ \opc \psi(t, \cdot) = 0\right)
%     \quad\implies\quad
%     \psi_0 = 0,
% \end{equation}
% \end{theorem}

\begin{remark}\label{rem}
% The result holds for any time $\tmax>0$.
The time $\tmax>0$ is not necessarily the duration of the full process, it can theoretically be chosen as small as desired. This property is called ``small time'' observability.
Even if the knowledge of the CLD at a fixed time $t$ is not sufficient to estimate the corresponding PSD, measuring the CLD on a small time interval $[t, t+dt]$ on which the process occurs is sufficient to estimate the PSD on this same interval.
\end{remark}

% \begin{itemize}
%     \item[Case 1.] There are two clusters of spheroidal crystals, sharing the same shape parameter $\ee_1 = \ee_2$, but with different grow rate, \emph{i.e.}, $G_1(t)\neq G_2(t)$ for some $t\in[0, \tmax]$.
%     This situation may occur if one wants to distinguish two different species of similar shape crystallizing at the same time in the reactor.
    
%     \item[Case 2.] There is a cluster of spherical crystals ($\ee_1 = 1$) and another cluster of prolate spheroids ($\eta_2>1$).
%     Their growth rate have constant ratio $\gamma$, \emph{i.e.}, $G_1(t) = \gamma G_2(t)$ for all $t\in[0, \tmax]$.
%     This situation frequently occurs in crystallization processes: there may be two facies of crystals of the same species appearing at the time in the reactor, one almost spherical, and another very elongated (see Figure~\ref{fig:petitspoiscarottes}).
% \end{itemize}
% \subsection{Proof of observability}

% Pushed to appendix.
\subsection{Numerical simulations}

For the numerical simulations, we consider the set of parameters given in Table~\ref{tab:param}.
\begin{table}[ht!]
    \centering
    \begin{tabular}{|c|c|c|c|}
    \hline
        $\rmin = 1.0\times 10^{-4}$m
        &$\rmax = 2.0\times 10^{-4}$m
        &$t_{\max} = 1$h
        &$N = 2$
        \\
        \hline
        $G_1 = 1.0\times 10^{-4}$m.h$^{-1}$
        &$G_2 = 2.0\times 10^{-4}$m.h$^{-1}$
        &$\ee_1 = 1$
        &$\ee_2 = 2$
        \\
        \hline
    \end{tabular}
    \caption{Parameters of the numerical simulation of the BFN algorithm.}
    \label{tab:param}
\end{table}
Simulations of \eqref{syst} and \eqref{obs2}-\eqref{obs2b} are performed with forward/backward finite differences, with spacing $dx = \frac{1}{100}$ for $\psi_1$ with growth rate $G_1$ and $dx = \frac{1}{50}$ for $\psi_2$ with growth rate $G_2$.
We fix $\psi_1 = \psi_2 = 0$ at the initial time $t = 0$, and choose the nucleation rates $u_1$ and $u_2$ such that, at time $t = 1$h, we have (see blue line on Figure~\ref{fig:bfn_sol10020})
\begin{equation}
    \psi_1(\tmax, r) = \psi_2(\tmax, r)
    =
    \frac{e^{-30(r-1.5\times 10^{-4})^2}}{\int_{1\times 10^{-4}}^{2\times 10^{-4}} e^{-30(\rho-1.5\times 10^{-4})^2} \diff \rho}.
\end{equation}
The BFN algorithm is initialized at $\hat\psi_1 = \hat\psi_2 = 0$.
On Figure~\ref{fig:bfn_sol10020}, we plot the estimations $\hat\psi_1$ and $\hat\psi_2$ obtained by BFN after $2n=$ 20 and 100 iterations.
After $20$ iterations, the shape of the two PSDs is already well estimated.
After $100$ iterations, the estimation of $\psi_2$ is far more accurate.
The error between the actual PSD and the estimation made by BFN, given by
\begin{equation}
    \|\eps^{2n}(t)\|^2_{L^2} = 
    \int_{1\times 10^{-4}}^{2\times 10^{-4}}
    \left(
    \psi_1(t, r) - \hat\psi_1^{2n}(t, r)
    \right)^2
    +
    \left(
    \psi_2(t, r) - \hat\psi_2^{2n}(t, r)
    \right)^2
    \diff r,
\end{equation}
is plotted in Figure~\ref{fig:bfn_err}. Applying a linear regression for $2n\geq30$, the rate of convergence is estimated as
$\|\eps^{2n}(t)\|_{L^2} \approx 0.156 \times 0.986^n$.

\begin{figure}[ht!]
    \centering
    \includegraphics[width=.5\linewidth]{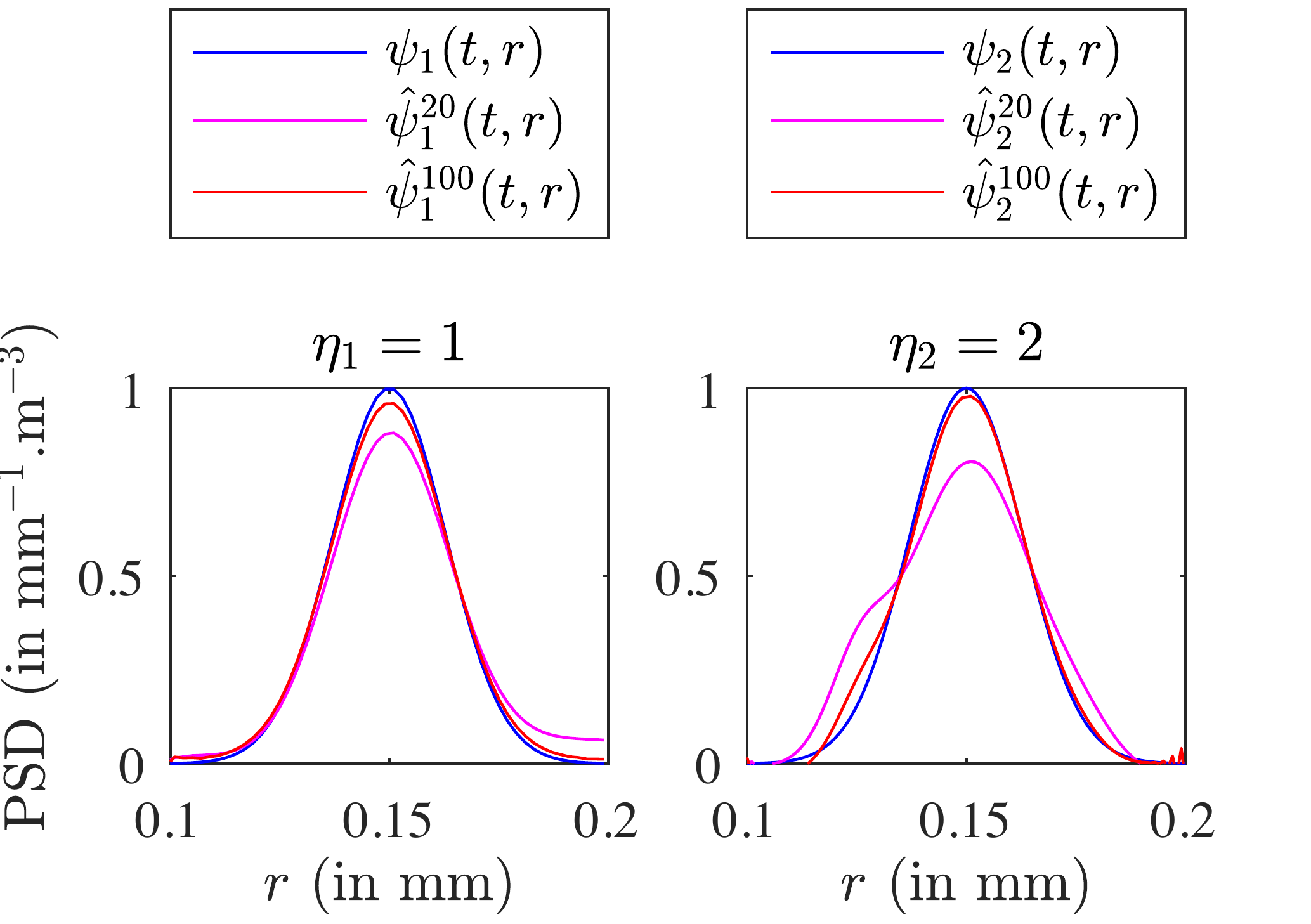}
    \caption{PSDs $\psi_1$ and $\psi_2$ at time $t=1$h and their estimation obtained by the BFN algorithm after $2n = 20$ and 100 iterations.}
    \label{fig:bfn_sol10020}
\end{figure}

\begin{figure}[ht!]
    \centering
    \includegraphics[width=.5\linewidth]{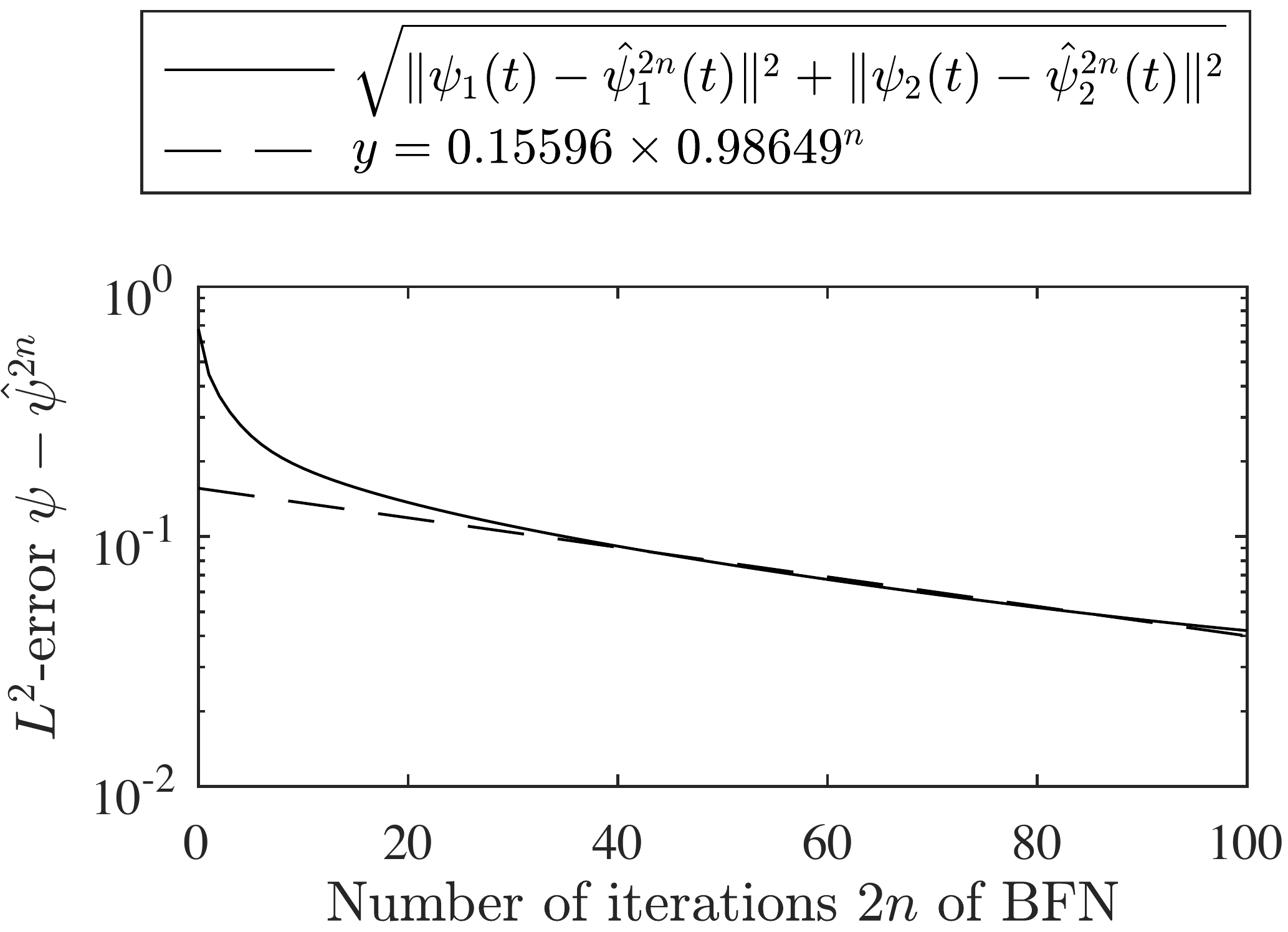}
    \caption{Evolution of the absolute error between the actual PSDs $\psi_1$ and $\psi_2$ at time $t = 1$h and the estimations $\hat\psi_1^{2n}$ and $\hat\psi_2^{2n}$ obtained by \eqref{obs2}-\eqref{obs2b}
    through iterations of the BFN algorithm.
    }
    \label{fig:bfn_err}
\end{figure}

% \begin{figure}[ht!]
%     \centering
%     \includegraphics[scale = 1]{bfnsol20.pdf}
%     \caption{PSDs $\psi_1$ and $\psi_2$ at time $t=1$h and their estimation obtained by the BFN algorithm after $2n = 20$ iterations.}
%     \label{fig:bfn_sol}
% \end{figure}

% \begin{figure}[ht!]
%     \centering
%     \includegraphics[scale = 1]{bfnsol.pdf}
%     \caption{PSDs $\psi_1$ and $\psi_2$ at time $t=1$h and their estimation obtained by the BFN algorithm after $2n = 100$ iterations.}
%     \label{fig:bfn_sol}
% \end{figure}

\section{Conclusion}

In this paper, a new expression of the CLD associated to a PSD of spheroid-like particles has been found.
Using this model,
two inversion procedures were proposed.
In the single-shape case, we proved the injectivity of the PSD-to-CLD operator. 
With a direct inversion method based on a Tikhonov regularization procedure, we were able to recover the PSD from the CLD in numerical simulations.
% First, Tikhonov regularization, which is already known to be efficient for other shapes of particles (\cite{}).
In the multi-shape case, we relied on an evolution model of a batch crystallization process to propose a new strategy for the CLD-to-PSD problem based on the
BFN algorithm.
We proved the convergence of the observer when two shapes of crystals coexist in the reactor: spheres and elongated spheroids, and numerically implemented this method.
The promising results obtained on simple numerically simulated PSDs, although preliminary,
suggest possible applications of these methods
to experimental data, that we hope will benefit from this theoretical study.

\subsubsection*{Acknowledgments}
The authors would like to thank V. Andrieu, É. Chabanon, É. Gagnière, U. Serres for many fruitful discussions.\\
This research was partially funded by the French Grant ANR ODISSE (ANR-19-CE48-0004-01).

\appendix

\section{Observability analysis}\label{app}

In this appendix, we prove the observability result.
\thobs*

Since $\tmax$ can be chosen as small as desired (see remark \ref{rem}), we actually show that if $\psi_0\neq 0$, the set of times $t\in[0, \tmax]$ such that $\opc \psi (t) \neq0$ is dense in $[0, \tmax]$.
Moreover,
$\opc\psi = \opc_1\psi_1+\opc_2\psi_2$,
where
\fonction{\mathcal{K}_i}{L^2((\rmin, \rmax); \R)}{L^2((0, \lmax); \R)}{\dtc}{\left(\ell\mapsto
\int_{\rmin}^{\rmax}\noy_i(\ell, r)\dtc(r)\diff r
\right)}

Hence, proving \eqref{eq:obs} is equivalent to proving that if $\psi_1(t, \cdot) = \opc_2\psi_2(t, \cdot)$ for all $t\in[0, \tmax]$,
then $\psi_1(t, r) = \psi_2(t, r) = 0$ for all $t\in[0, \tmax]$ and all $r\in[\rmin, \rmax]$.
% Hence, \eqref{eq:obs} is equivalent to
% \begin{multline}
%     \Big(\forall t\in[0, \tmax],\ \opc_1 \psi_1(t, \cdot) = \opc_2\psi_2(t, \cdot)\Big)\\
%     \quad\implies\quad\\
%     \Big(\forall t\in[0, \tmax],\forall r\in[\rmin, \rmax],\
%     \psi_1(t, r) = \psi_2(t, r) = 0
%     \Big)
% \end{multline}
The proof relies on properties of the successive derivatives of $\opc_i\psi_i$.

\medskip

Let $\mathcal{F}:L^2((\rmin, \rmax); \R)\to \R^{\N^*}$ be the linear map such that
$$
\mathcal{F}_n(\dtc) =\left(\mathcal F(\dtc)\right)_n
=
\int_{\rmin}^{\rmax}
    \frac{\dtc(r)}{r^{2n}}
\diff r ,
\qquad 
\forall \dtc\in L^2((\rmin, \rmax); \R), n\in \N^*.
$$
For all $\ee>0$, recall the definitions of section \ref{sec:inj}:
$$
a_n(\ee)=
\int_{\phi=0}^{2\pi}\int_{\theta=0}^\pi
\alpha_\ee^{n}(\phi,\theta) \frac{\sin \theta}{4\pi} \diff \theta\diff \phi,\qquad
b_n = \frac{(2n)!}{(n!)^2(1-2n)4^{2n}}.
$$
Let $\mathcal{A}(\ee)$ and $\mathcal B$ be the linear endomorphisms on $\R^{\N^*}$  such that, for any $\left(u_n\right)_{n\in \N^*}$
$$
\left(\mathcal{A}(\ee) u\right)_n= a_n(\ee) u_n, \qquad \left(\mathcal{B}  u\right)_n= b_n u_n.
$$
Then
$$
\left((\mathcal{K}_i\dtc)^{(2n)}(0)\right)_{n\in \N^*}
=
    \mathcal{B} \mathcal{A}(\ee_i) \mathcal{F}\dtc.
$$

\begin{lemma}[Asymptotic properties of $(a_n)$]\label{L:seq_moments}
The sequence $\left(a_n(\ee)\right)_{n\in \N}$ is such that 
$$
\lim_{n\to\infty} \frac{a_{n+1}(\ee)}{a_n(\ee)}
=
\begin{cases}
1 & \text{ if }\ee\geq 1,
\\
\dfrac{1}{\ee^2} & \text{ if }\ee< 1.
\end{cases}
$$
Furthermore, $a_n(\eta)>\sqrt{\pi/n}$ and if $\eta>1$ then $a_n(\eta)\to 0$.
\end{lemma}
\begin{proof}
Recall that 
$
\alpha_\ee(\phi,\theta)=\dfrac{\cos ^2\phi}{\cos^2\theta+\ee^2\sin^2 \theta}+\sin ^2\phi
$.
Then
$$
\|\alpha_\ee\|_\infty=\max_{
    \substack{
        \phi\in[0,2\pi] 
        \\ 
        \theta\in[0,\pi] 
        }
    }
\alpha_\ee(\phi,\theta)
=
\begin{cases}
1 & \text{ if }\ee\geq 1,
\\
\dfrac{1}{\ee^2} & \text{ if }\ee< 1.
\end{cases}
$$
Recall that $\frac{\sin\theta}{4\pi}$ is the density of a probability measure $\mu$ on $(\phi,\theta)\in [0,2\pi]\times [0,\pi]$. If we denote by $\mathbb{E}_\mu$ the expected value with respect to $\mu$, we obtain
$a_n(\ee)=\mathbb{E}_\mu\left(\alpha_\ee^n\right)$.
Then
$$
a_{n+1}(\ee)=\mathbb{E}_\mu \left(\alpha_\ee^{n+1}\right)\leq \|\alpha_\ee\|_\infty\mathbb{E}_\mu \left(\alpha_\ee^{n}\right)=\|\alpha\|_\infty a_{n}(\ee).
$$
On the other hand, $a_{n+1}(\ee)=\mathbb{E}_\mu \left(\left(\alpha_\ee^{n}\right)^{\frac{n+1}{n}}\right)$.
Notice that the function $x\mapsto x^{\frac{n+1}{n}}=x^{1+\frac{1}{n}}$ is convex. Hence, Jensen's inequality implies
$$
a_{n+1}(\ee)
=
\mathbb{E}_\mu \left(\left(\alpha_\ee^{n}\right)^{\frac{n+1}{n}}\right)
\geq 
\left(\mathbb{E}_\mu \left(\alpha_\ee^{n}\right)\right)^{1+\frac{1}{n}}
=
\left(a_{n}(\ee)\right)^{1+\frac{1}{n}}
$$
Thus $\left(a_{n}(\ee)\right)^{\frac{1}{n}}
\leq 
\frac{a_{n+1}(\ee)}{a_{n}(\ee)}
\leq \|\alpha_\ee\|_\infty
$.
Since $\mu$ is a probability measure,
$$
\left(a_{n}(\ee)\right)^{\frac{1}{n}}
=
\left(\mathbb{E}_\mu \left(\alpha_\ee^{n}\right)\right)^{\frac{1}{n}}
=
\left\|\alpha_\ee\right\|_{L^n(\mu)}
\xrightarrow[n \to \infty]{} 
\left\|\alpha_\ee\right\|_{L^\infty(\mu)}
=\left\|\alpha_\ee\right\|_{\infty},
$$
which concludes the proof of the first stated limit.

Regarding the supplementary asymptotic information, we first have 
naturally
$$
a_n(\eta)
\geq
\int_0^{2\pi}\sin^{2n}\phi \diff \phi 
=
2\pi \frac{(2n)!}{2^{2n}(n!)^2}
\sim 
2\sqrt{\frac{\pi}{n}}.
$$
The last limit  stated  is a consequence of Lebesgue's dominated convergence theorem, since 
$\alpha_\ee^{n}(\phi,\theta)\xrightarrow[n\to\infty]{}0$
for all $(\theta,\phi)$ such that $\phi\neq k\pi$, $k\in \Z$, (in which case $\alpha_\ee^{n}(\phi,\theta)=1$), and $0\leq \alpha_\ee^{n}(\phi,\theta)\leq 1$.
\end{proof}

Regarding the map $\mathcal{F}$, we have the following lemma.

\begin{lemma}\label{L:psi0_neq0}
Let $\dtc$ be continuous and such that $\dtc(\rmin)\neq 0$. Then
$$
\int_{\rmin}^{\rmax}\frac{\dtc(r)}{r^{2n}}\diff r\sim \frac{\dtc(\rmin)}{2n\rmin^{2n-1}}.
$$
\end{lemma}
\begin{proof}
Without loss of generality, we can assume that $\dtc(\rmin)>0$. Let $\mu\in (0,1)$, then by continuity of $\dtc$, there exists $R\in (\rmin,\rmax]$ such that for all $r\in [\rmin,R)$, $\dtc(r)\in  (\dtc(\rmin)(1-\mu),\dtc(\rmin)(1+\mu) )$.
Then
$$
\int_{\rmin}^{\rmax}\frac{\dtc(r)}{r^{2n}}\diff r
=\int_{\rmin}^{R}\frac{\dtc(r)}{r^{2n}}\diff r
+
\int_{R}^{\rmax}\frac{\dtc(r)}{r^{2n}}\diff r
$$
$$
\left|
    \int_{R}^{\rmax}\frac{\dtc(r)}{r^{2n}}\diff r
\right|
\leq 
 \frac{1}{2n-1}\frac{\left\|\dtc\right\|_\infty}{R^{2n-1}}.
$$
On the other hand, 
$$
\left|
    \int_{\rmin}^{R}\frac{\dtc(r)}{r^{2n}}\diff r-\frac{\dtc(\rmin)}{2n\rmin^{2n-1}}
\right|
\leq 
\frac{\dtc(\rmin)}{2n(2n-1)\rmin^{2n-1}}
+
\frac{\mu\dtc(\rmin)}{(2n-1)\rmin^{2n-1}}
+
\frac{\dtc(\rmin)(1+\mu)}{(2n-1)R^{2n-1}}.
$$
As a consequence,
$$
\begin{aligned}
\left|
\frac{2n\rmin^{2n+1}}{\dtc(\rmin)} \;
    \int_{\rmin}^{\rmax}\frac{\dtc(r)}{r^{2n}}\diff r-1
\right|
\leq 
&
\mu\frac{2n}{2n-1}
+
\frac{1}{ 2n-1 }
+
\frac{ (1+\mu)2n}{(2n-1)}\left(\frac{\rmin}{R}\right)^{2n-1}
\\
&+
\frac{2n}{2n-1}\frac{\left\|\dtc\right\|_\infty}{\dtc(\rmin)}
\left(\frac{\rmin}{\rmax}\right)^{2n-1}.
\end{aligned}
$$
The right-hand side has limit $\mu$ for any $\mu$ (independently of the value of $R$, which is always larger than $\rmin$), hence the left hand side has limit $0$.
\end{proof}

By integration by parts, we can obtain a corollary.
\begin{corollary}\label{C:dpsi0_neq0}

Let $\dtc$ be continuously differentiable and such that $\dtc(\rmin)=\dtc(\rmax)= 0$,$\dtc'(\rmin)\neq 0$. Then
$$
\int_{\rmin}^{\rmax}\frac{\dtc(r)}{r^{2n}}\diff r\sim \frac{\dtc'(\rmin)}{4n^2\rmin^{2n-1}}.
$$
\end{corollary}

These last two results allow to prove the following statement.

\begin{proposition}\label{P:border_obs}
There are no solutions  $\dtc_1,\dtc_2$ to
$
\mathcal{F}(\dtc_1)
=
\mathcal{A}(\eta)\mathcal{F}(\dtc_2)
$
(with $\psi_i(\rmax)=0$)
such that 
$
\dtc_1(\rmin),\partial_r\dtc_1(\rmin),\dtc_2(\rmin),\partial_r\dtc_2(\rmin)
$
are not all equal to $0$.

\end{proposition}
\begin{proof}
According to Lemmas~\ref{L:seq_moments}-\ref{L:psi0_neq0} and Corollary~\ref{C:dpsi0_neq0},
if $\dtc_1(\rmin)\neq 0$, then 
$
\rmin^{2n-1}\mathcal{F}_n(\dtc_i)$ converges to $0$ slower than
$
\rmin^{2n-1}
a_n(\eta)
\mathcal{F}_n(\dtc_2)$, since $a_n(\eta)\to 0$.
On the other hand, if $\dtc_1(\rmin)= 0$ then having $\dtc_2(\rmin)\neq 0$ implies that
$a_n(\eta)
\rmin^{2n-1}\mathcal{F}_n(\dtc_2)$ now converges slower than $\rmin^{2n-1}\mathcal{F}_n(\dtc_1)$ since $a_n(\eta)\geq 2\sqrt{\pi/n}$. Hence this implies that we must also have $\dtc_2(\rmin)=0$.
The same argument repeated on the derivatives yields the statement.
\end{proof}

In this first case, observability is proved by a sort of injectivity argument, the images of  $\mathcal{K}_1$ and $\mathcal{K}_2$ are such that their intersection cannot be reached through functions $\dtc$ that do not vanish at $\rmin$.

\begin{proposition}\label{P:second_der_obs}
Assume $\eta_1=1$  and $ \eta_2=\eta>1$. If $\dtc_1, \dtc_2$ are two non-zero solutions of the transport equation such that for some $\tau\in[0,\tmax]$, $\dtc_1(\tau),\dtc_2(\tau)\in H^2_0(\rmin,\rmax)$, then there exists no $\varepsilon>0$ such that 
$$
\mathcal{K}_1(\dtc_1(t))= \mathcal{K}_2(\dtc_2(t))
\qquad
\forall t\in (\tau-\varepsilon,\tau+\varepsilon)\cap [0,\tmax].
$$
\end{proposition}

\begin{proof}
By iterated integration by parts, for any $\dtc\in H^2_0(\rmin,\rmax)$
$$
\int_{\rmin}^{\rmax}
        \frac{\dtc''(r)}{r^{2n}}
\diff r
=
\left[
\frac{\dtc'(r)}{r^{2n}}
\right]_{\rmin}^{\rmax}
-
\left[2n
\frac{\dtc(r)}{r^{2n+1}}
\right]_{\rmin}^{\rmax}
+
2n(2n+1)
\int_{\rmin}^{\rmax}
        \frac{\dtc(r)}{r^{2n+2}}
\diff r.
$$
Hence for both $\dtc_i$, $i\in\{1,2\}$, at $t=\tau$,
$$
\mathcal{F}_n(\psi_i''(\tau))=2n(2n+1)\mathcal{F}_{n+1}(\psi_i(\tau)).
$$

We prove the result by contradiction. Assume there exists $\varepsilon>0$ such that 
$$
\mathcal{K}_1(\dtc_1(t))=\mathcal{K}_2(\dtc_2(t)), \qquad \forall t\in(\tau-\varepsilon,\tau+\varepsilon)\cap[0,\tmax],
$$
implies that 
\begin{equation}\label{E:seq_eq2}
    \mathcal{B} \mathcal{A}(\ee_1) \mathcal{F}\dtc_1(t)
    =
    \mathcal{B} \mathcal{A}(\ee_2) \mathcal{F}\dtc_2(t),
    \qquad \forall t\in (\tau-\varepsilon,\tau+\varepsilon)\cap[0,\tmax],
\end{equation}
and, term-wise,
$$
\mathcal{F}_n(\dtc_2(t))=\frac{a_n(\ee_1)}{a_n(\ee_2)}\mathcal{F}_n(\dtc_1(t)),
\qquad \forall t\in (\tau-\varepsilon,\tau+\varepsilon)\cap[0,\tmax],\forall n\in \N^*.
$$

On the other hand, equation~\eqref{E:seq_eq2} can be differentiated with respect to time. With 
$$
\frac{g_i}{G_i(t)}\frac{\partial}{\partial t}\frac{g_i}{G_i(t)}\frac{\partial}{\partial t} \dtc_i(t,r)=g_i^2\frac{\partial^2\dtc_i}{\partial r^2} (t,r)
\qquad \forall t\in [0,\tmax]
$$
hence, by the assumption that $g_1/G_1(t)=g_2/G_2(t)$,
$$
g_1^2 \mathcal{K}_1\left(\frac{\partial^2\dtc_1}{\partial r^2} (t)\right)
=
g_2^2 \mathcal{K}_2\left(\frac{\partial^2\dtc_2}{\partial r^2} (t)\right),
\qquad \forall t\in (\tau-\varepsilon,\tau+\varepsilon)\cap[0,\tmax].
$$
Likewise, this implies 
\begin{equation}\label{E:der_seq_eq2}
    g_1^2\mathcal{B} \mathcal{A}(\ee_1) \mathcal{F}\frac{\partial^2\dtc_1}{\partial r^2}(t)
    =
    g_2^2\mathcal{B} \mathcal{A}(\ee_2) \mathcal{F}\frac{\partial^2\dtc_2}{\partial r^2}(t),
    \qquad \forall t\in (\tau-\varepsilon,\tau+\varepsilon)\cap[0,\tmax],
\end{equation}
and, term-wise, 
$$
\mathcal{F}_n\left(\frac{\partial^2\dtc_2}{\partial r^2}(t)\right) = \frac{g_1^2}{g_2^2} \frac{a_n(\ee_1)}{ a_n(\ee_2)}\mathcal{F}_n\left(\frac{\partial^2\dtc_2}{\partial r^2}(t)\right)_n,
\qquad \forall t\in (\tau-\varepsilon,\tau+\varepsilon)\cap[0,\tmax],\forall n\in \N^*.
$$

Since equations~\eqref{E:seq_eq2}-\eqref{E:der_seq_eq2} hold, we have both
$$
\begin{aligned}
\mathcal{F}_n(\psi_1''(\tau))=2n(2n+1)\mathcal{F}_{n+1}(\psi_1(\tau)),
\\
\frac{g_1^2}{g_2^2}\frac{a_n(\eta_1)}{a_{n+1}(\eta_1)}\frac{a_{n+1}(\eta_2)}{a_{n}(\eta_2)}\mathcal{F}_n(\psi_1''(\tau))=2n(2n+1)\mathcal{F}_{n+1}(\psi_i(\tau)).
\end{aligned}
$$

If there isn't an infinity of non-zero terms in $\mathcal{F}\psi_1$, the function $\psi_1$ must be equal to zero since the family $(r\mapsto 1/r^{2n})_{n\in \N^*}$ is total. Assuming $\psi_1\neq 0$, then there is an infinity of non zero terms and, up to an extraction $(n_k)_{k\in \N^*}$ such that $\mathcal{F}_{n_k}(\psi_1)\neq 0$ for all $k\in \N^*$, and
\begin{equation}\label{E:quotient}
    \frac{g_1^2}{g_2^2}\frac{a_{n_k}(\eta_1)}{a_{n_k+1}(\eta_1)}\frac{a_{n_k+1}(\eta_2)}{a_{n_k}(\eta_2)}
    =
    1.
\end{equation}
If $\eta\geq 1$, $\frac{a_{n}(\eta)}{a_{n+1}(\eta)}\to 1$, hence \eqref{E:quotient} is is leading to an incoherent limit except in the case $g_1^2=g_2^2$.
However, since $a_{n}(1)=1$ and $\frac{a_{n+1}(\eta_2)}{a_{n}(\eta_2)}>1$, \eqref{E:quotient} cannot be satisfied termwise if   $g_1^2=g_2^2$.
\end{proof}

\begin{proposition}\label{propfinale}
Assume $\eta_1=1$  and $ \eta_2=\eta>1$. Let $\dtc_1, \dtc_2$ be two non-zero $H^2(\rmin,\rmax)$ solutions of their respective transport equations such that
$$
\dtc_i(\rmax,t)=0,\qquad \forall t\in [0,\tmax],i\in \{1,2\}.
$$
Then the set of times $t\in [0,\tmax] $ such that
$$
\mathcal{K}_1(\dtc_1(t))\neq \mathcal{K}_2(\dtc_2(t))
$$
is dense in $[0,\tmax]$.
\end{proposition}
\begin{proof}
Pick a time $t\in [0,\tmax]$. On the one hand, if $\dtc_i(\rmin,t)\neq 0$, or $\partial_r\dtc_i(\rmin,t)\neq 0$ for $i=1$ or $i=2$, then Proposition~\ref{P:border_obs} applies to prove that 
$
\mathcal{K}_1(\dtc_1(t))\neq \mathcal{K}_2(\dtc_2(t)).
$
On the other hand, if $\psi_1(t),\psi_2(t)\in H^2_0(\rmin,\rmax)$, then Proposition~\ref{P:second_der_obs} applies to prove that if $t$ is such that 
$
\mathcal{K}_1(\dtc_1(t))= \mathcal{K}_2(\dtc_2(t))
$
then any open interval containing $t$ must also contain a time $t'$ for which 
$
\mathcal{K}_{1}(\dtc_1(t'))\neq  \mathcal{K}_2(\dtc_2(t')).
$
This proves the statement.
\end{proof}

Proposition \ref{propfinale} implies Theorem~\ref{th:obs}, which concludes the observability analysis.

\bibliographystyle{abbrv}

\bibliography{references}

\end{document}